\documentclass[11pt]{smfart}
\usepackage{color}
\usepackage{amssymb,verbatim}
\usepackage{hyperref}
\usepackage{amsthm,array,amssymb,amscd,amsfonts,amssymb,latexsym, url}
\usepackage{amsmath}
\usepackage[all]{xy}
\usepackage[french]{babel}

\newcounter{spec}
{\end{list}}

\renewcommand{\P}{{\mathbf P}}

\newcommand{\N}{{\mathbb N}}
\newcommand{\Z}{{\mathbb Z}}
\newcommand{\Q}{{\mathbb Q}}
\newcommand{\C}{{\mathbb C}}

\renewcommand{\L}{{\mathbb L}}
\newcommand{\oi}{\hskip1mm {\buildrel \simeq \over \rightarrow} \hskip1mm}

\newcommand{\Br}{{\operatorname{Br \  }}}
\newcommand{\Ker}{{\operatorname{Ker}}}

\newcommand{\Hom}{{\operatorname{Hom}}}
\newcommand{\Cores}{{\operatorname{Cores}}}
\newcommand{\Spec}{{\operatorname{Spec \ }}}

\newcommand{\cd}{{\operatorname{cd }}}

\newcommand{\Norm}{{\operatorname{Norm}}}

\renewcommand{\lim}{\varprojlim}

\numberwithin{equation}{section}

\newfont{\gothic}{eufb10}

\renewcommand{\qed}{{\hfill$\square$}}

\newtheorem{theo}{Th\'{e}or\`{e}me}[section]
\newtheorem{prop}[theo]{Proposition}

\newtheorem{lem}[theo]{Lemme}
\newtheorem{cor}[theo]{Corollaire}
\theoremstyle{definition}
\newtheorem{defi}[theo]{D\'efinition}
\theoremstyle{remark}
\newtheorem{rema}[theo]{Remarque}

\newtheorem{ex}[theo]{Exemples}

\newcommand{\bthe}{\begin{theo}}
\newcommand{\ble}{\begin{lem}}
\newcommand{\bpr}{\begin{prop}}
\newcommand{\bco}{\begin{cor}}
\newcommand{\bde}{\begin{defi}}
\newcommand{\ethe}{\end{theo}}
\newcommand{\ele}{\end{lem}}
\newcommand{\epr}{\end{prop}}
\newcommand{\eco}{\end{cor}}
\newcommand{\ede}{\end{defi}}

\newcommand{\et}{{\operatorname{\acute{e}t}}}

\newcommand{\F}{{\mathbb F}}

\newcommand{\G}{{\mathbb G}}

\DeclareFontFamily{U}{wncy}{}
\DeclareFontShape{U}{wncy}{m}{n}{%
<5>wncyr5%
<6>wncyr6%
<7>wncyr7%
<8>wncyr8%
<9>wncyr9%
<10>wncyr10%
<11>wncyr10%
<12>wncyr6%
<14>wncyr7%
<17>wncyr8%
<20>wncyr10%
<25>wncyr10}{}
\DeclareMathAlphabet{\cyr}{U}{wncy}{m}{n}

\begin{document}

  \title[Lois de r\'eciprocit\'e sup\'erieures et points rationnels]
 {Lois de r\'eciprocit\'e sup\'erieures et points rationnels}

\author{J.-L. Colliot-Th\'el\`ene}
\address{C.N.R.S., Universit\'e Paris Sud\\Math\'ematiques, B\^atiment 425\\91405 Orsay Cedex\\France}
\email{jlct@math.u-psud.fr}
\author{R. Parimala}
\address{Department of Mathematics {\&} CS \\
Emory University, Atlanta, Georgia  \\ USA}
\email{parimala@mathcs.emory.edu}
\author{V. Suresh}
\address{Department of Mathematics {\&} CS \\
Emory University, Atlanta, Georgia \\ USA}
\email{suresh.venapally@gmail.com}

\date{}
\maketitle

 \begin{abstract}
Soit $K=\C((x,y))$  ou $K=\C((x))(y)$.   Soit $G$ un $K$-groupe alg\'ebrique lin\'eaire connexe.
Il a \'et\'e \'etabli que si $G$ est $K$-rationnel, c'est-\`a-dire de corps des fonctions transcendant pur sur 
$K$,
si un espace principal homog\`ene sous $G$ a des points rationnels dans tous
les compl\'et\'es de $K$ par rapport aux valuations de $K$, alors il a
un point rationnel.
Nous montrons ici qu'en g\'en\'eral l'hypoth\`ese de $K$-rationalit\'e
ne peut \^etre omise. Nous utilisons pour cela une  obstruction d'un nouveau
type, fond\'ee sur les lois  de r\'eciprocit\'e sup\'erieure sur un sch\'ema de dimension deux. Nous donnons aussi  une famille
d'espaces principaux homog\`enes pour laquelle cette  obstruction raffin\'ee \`a l'existence d'un point rationnel
est la seule obstruction.
 \end{abstract}

\begin{altabstract}
Let  $K=\C((x,y))$  or $K=\C((x))(y)$. Let $G$ be a connected linear algebraic group over $K$.
Under the assumption that the $K$-variety $G$ is $K$-rational, i.e. that the function field
is purely transcendental, it was proved that a principal homogeneous space of $G$
has a rational point over $K$ as soon as it has one over each completion of $K$
with respect to a valuation.  In this paper we show that one cannot in general do without
the $K$-rationality assumption. To produce our examples, we introduce a new type of
obstruction. It is based on higher reciprocity laws on a 2-dimensional scheme. We also produce a  family of principal
homogeneous spaces for which the refined obstruction controls exactly the existence of
rational points.

\end{altabstract}

  \tableofcontents

\section{Introduction}

\subsection{Le cadre}

Soient  $\mathcal X$ un sch\'ema r\'egulier int\`egre de dimension 2, 
  $R$ un anneau  local int\`egre,  hens\'elien, excellent, 
de corps r\'esiduel $k$
et $p : {\mathcal X} \to \Spec R$ un morphisme projectif surjectif
satisfaisant l'une des conditions suivantes.

(a)  L'anneau $R$ est un anneau de valuation discr\`ete,
 les fibres de $p$
sont de dimension 1, la fibre g\'en\'erique est lisse et g\'eom\'etriquement int\`egre.
On appellera
cela le cas \og semi-global  \fg.

(b)  L'anneau  $R$ est  de dimension 2, et $p$ est birationnel. On appellera
cela le cas \og  local  \fg.

On note $O$ le point ferm\'e de $\Spec R$.
Si $p : {\mathcal X} \to \Spec R$ n'est pas un isomorphisme,
la fibre sp\'eciale $X_{0}=p^{-1}(O)$  est une courbe projective, en g\'en\'eral 
r\'eductible, sur le corps $k$.

 Soit $K$ le corps des fonctions rationnelles de $\mathcal X$. 
Soit  $\Omega=\Omega_{K}$ la famille des valuations discr\`etes de rang 1 sur   $K$.
  Pour $v \in \Omega$, on note $K_{v}$ le hens\'elis\'e de $K$ en $v$
 et $R_{v}$ son anneau des entiers.  
 Tant dans le cas \og semi-global \fg\ que dans le cas \og local \fg, on a la propri\'et\'e suivante 
des groupes de Brauer :

\smallskip

{\it  L'application
$\Br K \to \prod_{v\in \Omega} \Br K_{v}$
est  injective.}

\smallskip

Cette propri\'et\'e, valable quel que soit le corps r\'esiduel $k$,  semble connue de plusieurs auteurs. 
Dans le cas \og semi-global \fg\  et $R$ complet, on peut renvoyer \`a \cite[Thm. 4.3]{CTPaSu}.
Pour $R$ hens\'elien, on peut l'\'etablir dans les deux cas par une 
  combinaison de  \cite[Thm. 1.8]{CTOjPa}
 et \cite[Prop. 1.14]{CTOjPa}, comme expliqu\'e dans le cas \og local \fg\ dans \cite[\S 3]{Hu1}.
 
Ceci donne donc un analogue du th\'eor\`eme de Hasse, Brauer, Noether en th\'eorie du corps de classes.

Les probl\`emes suivants, 
analogues  de probl\`emes   r\'esolus
dans la situation  o\`u $K$ est un corps global et $\Omega$ l'ensemble de ses places (\cite{sansuc}, \cite{borovoi}),
ont fait l'objet d'un certain nombre d'\'etudes, que nous mentionnons plus bas.

\medskip

(1) Soit $G$ un $K$-groupe lin\'eaire lisse.
 L'application naturelle $$H^1(K,G) \to \prod_{v \in \Omega}  H^1(K_{v},G)$$
a-t-elle un noyau trivial ? En d'autres termes, tout espace principal homog\`ene (torseur) sous $G$ qui a des points
dans tous les $K_{v}$  pour $v\in \Omega$ a-t-il un $K$-point ?

(2) Soit $\mu$ un $K$-module galoisien fini. Soit $i\geq 1$ un entier.
L'application 
$$H^{i}(K,\mu)  \to \prod_{v\in \Omega} H^{i}(K_{v},\mu)$$
est-elle injective ?

(3)   Soit $Z$ une $K$-vari\'et\'e projective et lisse, g\'eom\'etriquement connexe,  et qui est  un
espace homog\`ene d'un $K$-groupe lin\'eaire $G$. Si l'on a $Z(K_{v})$ non vide pour tout $v \in \Omega$, a-t-on $Z(K)$ non vide ?

\medskip

Sans hypoth\`ese sur le corps $k$, dans le cas \og semi-global \fg, 
 on peut citer  des travaux de  M. Artin (voir Grothendieck \cite[III, \S 3]{GBBr}),
 des auteurs  et M. Ojanguren \cite{CTOjPa}, de Harbater, Hartmann et Krashen 
  \cite{HHK1,HHK2, HHK3} et des auteurs \cite{CTPaSu}. 
  Dans le cas \og local \fg, on peut citer
   \cite{CTOjPa}  et Y. Hu  \cite{Hu1}.

Lorsque le corps  $k$ est  s\'eparablement clos,  
le corps $K$ est alors un corps de dimension cohomologique 2,
et les probl\`emes ci-dessus sont tr\`es proches des probl\`emes \'etudi\'es
sur les corps globaux.  On consultera le rapport g\'en\'eral \cite{parimala}.
Tant dans le cas \og  semi-global \fg\ que dans le cas \og local  \fg\,
on sait  montrer dans de nombreux cas
que l'on a $H^1(K,G)=0$ (cf. \cite[Thm. 1.2, Thm. 1.4]{CTGiPa})
pour  $G$ semi-simple simplement connexe,  ce qui ram\`ene
  le probl\`eme (1) 
  pour $G$ semi-simple
   au probl\`eme (2)
pour $i=2$.
 Dans le cas \og local \fg,
 on peut citer des travaux de Artin, Ford, Saltman,  
 des auteurs et M. Ojanguren \cite{CTOjPa}, des
 auteurs  et  P. Gille \cite{CTGiPa}.

Lorsque le corps  $k$ est fini, dans le cas \og semi-global \fg,
 le probl\`eme (2) a une longue histoire. On se restreignait classiquement aux valuations
 triviales sur l'anneau~$R$.
Pour $H^2(K,\mu_{n})$,
  l'injectivit\'e est une variante du th\'eor\`eme
de dualit\'e de Tate et Lichtenbaum sur les vari\'et\'es ab\'eliennes sur les
corps $p$-adiques. Pour $H^3(K,\mu_{n}^{\otimes 2})$,
 l'injectivit\'e est  essentiellement un r\'esultat de Kato \cite{kato}.
Lorsque l'on prend toutes les valuations discr\`etes, 
 on conjecture (\cite{CTPaSu}  dans le cas \og semi-global \fg) que pour $G$ semi-simple
simplement connexe, la question (1) a une r\'eponse
affirmative. Ceci a \'et\'e \'etabli dans de nombreux cas
(\cite{CTPaSu}, \cite{Hu3}, \cite{preeti}). La question (3)
fait aussi l'objet d'une conjecture. Pour les quadriques,
des r\'esultats sont obtenus dans
\cite{CTOjPa}, \cite{Hu3}. 
Un certain nombre de ces r\'esultats utilisent l'invariant de Rost pour
se ramener \`a l'\'enonc\'e de Kato sur $H^3(K,\mu_{n}^{\otimes 2})$.
Dans le cas \og local \fg, on peut citer \cite{CTOjPa}, \cite{Hu2}, \cite{Hu3}.

En ce qui concerne la question (2),  
on a donn\'e des exemples (\cite[Rem. 3.1.2]{CTOjPa}, \cite[\S 3.3, Rem. 2]{CTGiPa} dans le cas \og local \fg,  \cite[\S 6]{CTPaSu} dans le cas \og semi-global \fg) pour lesquels l'application
$$H^{1}(K,\Z/2)  \to \prod_{v\in \Omega} H^{1}(K_{v},\Z/2)$$
n'est pas injective -- \`a la diff\'erence de la situation sur un corps de nombres.
Dans ces exemples,   
le graphe des composantes
de la fibre sp\'eciale g\'eom\'etrique n'est pas un arbre : il contient un lacet.
Pour plus de d\'etails, voir  \cite{HHK2}.  
L'\'enonc\'e sur le groupe de Brauer montre que la question (2) pour $i=2$
et $\mu=\mu_{n}$, avec $n\geq 1$ entier inversible sur $\mathcal X$, a une r\'eponse affirmative.
Dans le cas \og semi-global \fg, 
en \'egale caract\'eristique,
Harbater, Hartmann et Krashen \cite{HHK3} viennent d'\'etendre ce r\'esultat \`a tout $H^{i}(K,\mu_{n}^{\otimes (i-1)})$ pour $i \geq 2$.

Les exemples mentionn\'es \`a l'instant montrent qu'il convient de restreindre la question (1) au cas o\`u $G$
est connexe. Supposons que  $k$ est 
alg\'ebriquement clos
de caract\'eristique z\'ero.
Dans le cas \og local \fg,  
on a montr\'e  (\cite[Thm. 5.2 (b)  (ii)]{CTGiPa}, \cite[Thm. 7.9]{BoKu}) que
la r\'eponse  \`a (1) est affirmative si $G$ est un $K$-groupe lin\'eaire {\it connexe $K$-rationnel},
i.e. de corps des
fonctions transcendant pur sur $K$.
Supposons de plus $R$ complet. Dans le cas  \og  semi-global \fg,
Harbater, Hartmann et Krashen \cite[Thm. 8.10]{HHK2} ont   \'etabli   que
si $G$ est un $K$-groupe lin\'eaire {\it connexe $K$-rationnel} , alors la r\'eponse \`a la question (1)  est  affirmative.
Ce r\'esultat est une cons\'equence du th\'eor\`eme \og local \fg\ mentionn\'e ci-dessus et
d'un th\'eor\`eme local-global vis-\`a-vis d'un
autre ensemble de surcorps de $K$, th\'eor\`eme fondamental des m\^emes auteurs
\cite[Thm. 3.7]{HHK1},  valable  pour tout groupe $K$-{\it rationnel} sans aucune
hypoth\`ese sur le corps r\'esiduel $k$.

\medskip

Pour $G/K$ semi-simple simplement connexe, non n\'ecessairement $K$-rationnel,
on a pu dans certains cas utiliser l'invariant de Rost pour donner une r\'eponse
positive \`a la question (1) : voir \cite[\S 5]{CTPaSu}, 
\cite{Hu3}, \cite{preeti}, \cite{HHK3}.

\medskip

{\it Tous ces r\'esultats, analogues d'\'enonc\'es bien connus 
pour les groupes lin\'eaires sur un corps de nombres  \cite{sansuc},
 laissaient
  ouvertes la question (1) 
pour les groupes  lin\'eaires  connexes quelconques, et, d\'ej\`a pour $i=2$, la question (2) sur les
modules galoisiens finis.}

\subsection{Les r\'esultats du pr\'esent article}

{\it Nous montrons que d\'ej\`a dans le cas o\`u le corps r\'esiduel $k$ est alg\'ebriquement clos,
tant dans le cas \og semi-global \fg\  que dans le cas \og local \fg,
il existe des $K$-groupes $G$ connexes pour lequel le principe local-global pour les
espaces principaux homog\`enes, par rapport
aux places dans $\Omega$, est en d\'efaut : la question (1) a en g\'en\'eral une r\'eponse
n\'egative.  }

Dans le cas \og local \fg, ceci r\'epond \`a une question pos\'ee il y a dix ans (\cite[\S3.4, Question]{CTGiPa}).
Dans le cas  \og semi-global \fg, ceci r\'epond \`a une question motiv\'ee par \cite{HHK1} 
et   pos\'ee explicitement dans~\cite{CTPaSu}.

{\it Dans le cas \og semi-global \fg, nous donnons aussi de tels exemples avec $k$ un corps fini,
par exemple avec $K$ le corps des fonctions d'une courbe sur un corps $p$-adique.}

Nous commen\c cons par donner un tel exemple avec $G$  un $K$-tore 
dont le groupe de Brauer non ramifi\'e est non trivial, ce qui implique
en particulier que la $K$-vari\'et\'e $G$ n'est pas $K$-rationnelle.

Des techniques connues permettent alors de donner un exemple de module galoisien fini
$\mu$ tel que l'application
$H^{2}(K,\mu)  \to \prod_{v\in \Omega} H^{2}(K_{v},\mu)$
n'est pas injective (r\'eponse n\'egative \`a la question (2)).
{\it La pr\'esence d'un lacet dans la fibre sp\'eciale joue un r\^ole-cl\'e dans la construction
des exemples.}
Dans le cas \og local   \fg, ceci  r\'esoud une autre 
question pos\'ee il y a dix ans 
 (\cite[{\it ibidem}]{CTGiPa}).

{\sl Mutatis mutandis}, une m\'ethode de  Serre (pour $K$ corps global)
 permet alors  de construire un $K$-groupe $G$
semi-simple connexe (non simplement connexe) pour lequel la fl\`eche diagonale
$H^1(K,G) \to \prod_{v \in \Omega}  H^1(K_{v},G)$
a un noyau non trivial.
 
 Lorsqu'on dispose d'exemples sur un corps $K'$ extension finie d'un corps $K$,
la restriction des scalaires \`a la Weil permet de fabriquer des exemples
sur le corps $K$. On peut en fin de compte donner des exemples,
dans le cas \og  local \fg,  sur  $K=\C((X,Y))$, corps des fractions du corps
des s\'eries formelles $\C[[X,Y]]$ sur le corps des complexes,
et  dans le cas \og semi-global \fg, sur $K=\C((X))(Y)$.

\bigskip

Pour obtenir nos exemples,
nous
 utilisons de fa\c con concr\`ete
 un {\it nouveau type d'obstruction au principe de Hasse
sur le corps des fonctions de sch\'emas r\'eguliers  int\`egres de dimension quelconque}, 
propos\'ee 
 par Colliot-Th\'el\`ene il y a quelques
ann\'ees sur le mod\`ele de l'obstruction de Brauer-Manin sur les corps
de fonctions d'une variable sur un corps fini.

Cette obstruction exploite  les lois de r\'eciprocit\'e  pour la cohomologie galoisienne
des modules finis constants,
aux points ferm\'es de tels sch\'emas, comme fournies par la th\'eorie
de Bloch-Ogus \cite{BO},
d\'evelopp\'ee   par Kato \cite{kato} puis par Jannsen et Saito \cite{JaSa} en situation d'in\'egale caract\'eristique.
La construction g\'en\'erale  et les propri\'et\'es de telles obstructions sont donn\'ees  au \S \ref{blochogusrecip}.

Le cas particulier de cette obstruction utilis\'e pour nos exemples est le suivant.
Etant donn\'es un sch\'ema $\mathcal X$ r\'egulier int\`egre de dimension 2
de corps des fractions $K$, avec $2$ inversible sur $X$,
  une $K$-vari\'et\'e $Z$ projective, lisse, g\'eom\'etriquement int\`egre
  et un \'el\'ement du groupe de cohomologie non ramifi\'ee $H^{2}_{nr}(K(Z)/K,\Z/2)$,
   la combinaison des  lois de r\'eciprocit\'e sur 
   $\mathcal X$ en les diff\'erents points ferm\'es $M$ de  $\mathcal X$  permet de d\'efinir une obstruction \'eventuelle
   \`a la propri\'et\'e suivante :  si $Z$ admet des points dans tous les hens\'elis\'es
   $K_{\gamma}$ pour $\gamma$ courbe int\`egre sur $\mathcal X$,
   alors $Z$ admet un $K$-point.

   Les $K$-tores  
    utilis\'es dans nos exemples initiaux  sont  donn\'es par une \'equation
    $$(X_{1}^2-aY_{1}^2)(X_{2}^2-bY_{2}^2)(X_{3}^2-abY_{3}^2)=1$$
   avec $a,b \in K^{\times}$.  Tout espace homog\`ene 
   principal
   $E$ sous un tel $K$-tore
   est donn\'e par une \'equation
   $$(X_{1}^2-aY_{1}^2)(X_{2}^2-bY_{2}^2)(X_{3}^2-abY_{3}^2)=c$$
   avec $c \in K^{\times}$. On a montr\'e dans \cite{CT12}
   que pour $Z$ une $K$-compactification projective
    et lisse de la $K$-vari\'et\'e $E$, le quotient $\Br Z/ \Br K$
    est d'ordre au plus 2, et on a donn\'e un g\'en\'erateur explicite $A \in {}_{2}\Br Z=H^2_{nr}(K(Z)/K,\Z/2)$.
   Ceci est rappel\'e au \S  \ref{algebre}. Au \S  \ref{evallocaledeA}, nous discutons l'existence  de points
   rationnels de $E$ sur les hens\'elis\'es de $K$, et nous calculons les valeurs prises par $A$
   sur ces points, ce qui nous permet de calculer l'obstruction de r\'eciprocit\'e sup\'erieure associ\'ee \`a $A$
   en un point ferm\'e   du sch\'ema $\mathcal X$. Nous \'etudions aussi le comportement
   de cette obstruction apr\`es \'eclatement d'un point de $\mathcal X$.

\medskip
   
  Nous sommes alors en mesure de construire, au  \S \ref{coronidisloco},  les exemples annonc\'es.

\medskip

Au \S \ref{paradescente}, sous l'hypoth\`ese que le corps $k$
est s\'eparablement clos, nous montrons que si une vari\'et\'e $Z$
du type ci-dessus a des points dans tous les hens\'elis\'es $K_{\gamma}$,
et s'il existe un bon mod\`ele $\mathcal X$ sur lequel la classe $ A \in {}_{2}\Br Z$
ne donne pas d'obstruction de r\'eciprocit\'e sup\'erieure, {\it alors $Z$ poss\`ede 
un $K$-point}.
Nous op\'erons pour ce faire une
 {\it descente d'un nouveau type}
 au-dessus du corps des fonctions de tels sch\'emas.

On peut ainsi  se demander si le r\'esultat obtenu pour les espaces principaux homog\`enes
des $K$-tores du type ci-dessus s'\'etend aux espaces principaux homog\`enes sous un
$K$-tore quelconque.

 \medskip

 \`A la diff\'erence de notre pr\'ec\'edent travail \cite{CTPaSu}, le pr\'esent article n'utilise 
 pas    les travaux de Harbater, Hartmann et Krashen, si ce n'est, au paragraphe \ref{traduction},
 pour traduire nos r\'esultats dans leur cadre.

 \bigskip
 
 Une valuation discr\`ete sur un corps $K$  est ici une valuation discr\`ete de rang~1, et le groupe
 de la valuation est $\Z$.
 
Si une telle valuation discr\`ete est hens\'elienne, et que l'on note $\hat{K}$ le compl\'et\'e de $K$,
pour toute $K$-vari\'et\'e lisse $Z$, les conditions $Z(K)\neq \emptyset$ et $Z(\hat{K})\neq \emptyset$
sont \'equivalentes.

Dans tout cet article, sauf mention du contraire, la cohomologie employ\'ee est la cohomologie
\'etale des sch\'emas, qui se sp\'ecialise \`a la cohomologie galoisienne sur les corps.

Pour $n>0$ un entier et $A$ un groupe ab\'elien, on note  ${}_{n}A$ le sous-groupe form\'e des \'el\'ements
de $A$ annul\'es par $n$.

\section{Complexe de Bloch-Ogus et obstructions de r\'eciprocit\'e}\label{blochogusrecip}

\subsection{Complexe de Bloch-Ogus--Kato}\label{BOK}

 Nous commen\c cons par quelques rappels sur la th\'eorie de Bloch-Ogus \cite{BO},
 en nous limitant au cadre qui nous int\'eresse ici.
 La th\'eorie est bien document\'ee pour les vari\'et\'es lisses sur un corps,
 il est plus d\'elicat de trouver des r\'ef\'erences pour un sch\'ema r\'egulier
 quelconque. Il y a \`a cela deux raisons : 
 on a  besoin de th\'eor\`emes de puret\'e pour la cohomologie \'etale,
  et pour aller plus loin
 on a besoin de conna\^{\i}tre  la conjecture de Gersten.
 La puret\'e a \'et\'e \'etablie par Gabber \cite[Thm. 3.1.1]{surGab}. Mais la conjecture de Gersten
 n'est toujours pas connue dans le cadre r\'egulier quelconque.
 
Soit   $\mathcal X$ un sch\'ema excellent int\`egre de dimension $d$, tel que
pour tout ferm\'e irr\'eductible $F \subset \mathcal X$,
on a la propri\'et\'e ${\dim}(F) = d - {\rm codim}_{\mathcal X}(F)$.

On note $K$ le corps 
de 
fonctions de $\mathcal X$.
On note ${\mathcal X}^{(i)}$ l'ensemble des points de codimension $i$
de $\mathcal X$, et on note $\kappa(x)$ le corps r\'esiduel en un point $x \in {\mathcal X}$.
Soit $n>0$ un entier inversible sur $\mathcal X$.
Pour $i>0$, on note $\mu_{n}^{\otimes i}$ le produit tensoriel
$i$-fois du faisceau \'etale $\mu_{n}$ (racines $n$-i\`emes de $1$) avec lui-m\^{e}me,
on note $\mu_{n}^{0}=\Z/n$ et, pour $i<0$, on d\'efinit $\mu_{n}^{\otimes i}=
\Hom(\mu_{n}^{\otimes (-i)}, \Z/n)$.

Pour tout  $r \in \N$ et tout $i \in \Z$, on a le  {\it complexe} de Bloch-Ogus arithm\'etique $C_{r,i}$  :
$$0 \to  H^{r}(K,\mu_{n}^{\otimes i}) \to \bigoplus_{\gamma \in {\mathcal X}^{(1)} } H^{r-1}(\kappa(\gamma),  \mu_{n}^{\otimes (i-1)})     \to 
\bigoplus_{M \in {\mathcal X}^{(2)} } H^{r-2}(\kappa(M),  \mu_{n}^{\otimes (i-2)})  
\to \cdots $$
o\`u les degr\'es croissent vers la droite, le  groupe $H^{r}(K,\mu_{n}^{\otimes i})$ \'etant plac\'e en degr\'e~0.
Ces complexes sont construits par Kato \cite[\S 1, Prop. 1.7]{kato} (avec une autre graduation)  et \'etudi\'es de fa\c con syst\'ematique
par Jannsen et Saito  \cite{JaSa}.
 Les applications r\'esidus
$$\partial_{\gamma}   : H^{r}(K,\mu_{n}^{\otimes i}) \to H^{r-1}(\kappa(\gamma),  \mu_{n}^{\otimes (i-1)})$$
  pour $\gamma \in  {\mathcal X}^{(1)} $  et
  $$   \partial_{\gamma,M} : H^{r-1}(\kappa(\gamma),  \mu_{n}^{\otimes (i-1)}) \to H^{r-2}(\kappa(M),  \mu_{n}^{\otimes (i-2)})$$
 pour $\gamma \in  {\mathcal X}^{(1)} $ et $M \in   {\mathcal X}^{(2)} $
sont obtenues par normalisation et sommation
\`a partir des r\'esidus classiques sur un anneau de valuation discr\`ete.

Lorsque $\mathcal X$ est un sch\'ema r\'egulier, hypoth\`ese que nous faisons d\'esormais,
sous certaines hypoth\`eses, on peut donner des informations sur
les groupes d'homologie du complexe ci-dessus : voir \cite{kato} et 
\cite{sasa}.

\medskip

Supposons de plus  ${\mathcal X}$ de dimension 2.
Consid\'erons le complexe $C_{2,2}$ :
$$ 0 \to H^2(K,\mu_{n}^{\otimes 2}) \to
 \bigoplus_{\gamma \in {\mathcal X}^{(1)}} \kappa(\gamma)^{\times}/\kappa(\gamma)^{\times n}
\to \bigoplus_{M\in {\mathcal X}^{(2)} } \Z/n \to 0,$$
o\`u les trois termes du milieu sont, de gauche \`a droite, plac\'es en degr\'e 0, 1, 2.

On a $H^2(C_{2,2}) =  CH_{0}({\mathcal X})/n$,
o\`u $CH_{0}(Y)$ d\'esigne le groupe de Chow des cycles de dimension z\'ero sur 
un sch\'ema $Y$.

Sous l'hypoth\`ese que l'on a $\Z/n \oi \mu_{n}$ sur $\mathcal X$, 
le groupe $H^0(C_{2,2})$,
par un r\'esultat de puret\'e de Auslander--Goldman et Grothendieck \cite[II, Prop. 2.3]{GBBr},
est  isomorphe au sous-groupe de  $n$-torsion du groupe de Brauer de $\mathcal X$.

\begin{prop}\label{blochogus}
Soient  $\mathcal X$ un sch\'ema r\'egulier int\`egre de dimension 2, $R$ 
 un anneau  local int\`egre,  hens\'elien, excellent, 
de corps r\'esiduel $k$, et  $p :  {\mathcal X} \to \Spec R$ un morphisme projectif.
Supposons que l'on est dans l'un des cas suivants.

(a) L'anneau $R$ est un anneau de valuation discr\`ete,
 les fibres de $ {\mathcal X} \to \Spec R$
sont de dimension 1, la fibre g\'en\'erique est lisse et g\'eom\'etriquement int\`egre.

(b) L'anneau  $R$ est  de dimension 2, et $p$ est birationnel.

 Soit $K$ le corps des fonctions rationnelles de $\mathcal X$. 
 Soit $n>0$ un entier inversible dans $k$.
 Supposons  donn\'e un isomorphisme $\Z/n \oi \mu_{n}$ sur $\mathcal X$.
 
 \smallskip
 
 (i) On a un complexe naturel 
$$0 \to {}_{n} \Br K \to \bigoplus_{\gamma \in {\mathcal X}^{(1)}} \kappa(\gamma)^{\times}/\kappa(\gamma)^{\times n}
\to \bigoplus_{M \in {\mathcal X}^{(2)}} \Z/n \to 0$$
covariant  en $\mathcal X$ par morphisme propre de $R$-sch\'emas.

\smallskip

(ii) Le complexe est exact en son troisi\`eme terme.

Si $k$ est un corps s\'eparablement clos ou un corps fini, le complexe
est exact en son premier terme.

Si $k$ est s\'eparablement clos et la conjecture de Gersten vaut pour $\mathcal X$
et le faisceau $\mu_{n}$, alors ce complexe est une suite exacte.

Si $k$ est un corps fini et  si la conjecture de Gersten vaut pour $\mathcal X$
et le faisceau $\mu_{n}$, alors ce complexe est exact sauf au terme m\'edian,
o\`u l'homologie est un sous-groupe de $\Z/n$.
 \end{prop}

\begin{proof}

Pour le point (i), en particulier pour la fonctorialit\'e covariante par morphisme propre,
nous renvoyons le lecteur \`a \cite{JaSa}. Voir aussi la remarque \ref{foncpropre} ci-dessous.

Comme $R$ est hens\'elien, on a   $CH_{0}({\mathcal X})=0$. Ceci assure la nullit\'e de 
l'homologie du complexe en le terme $ \bigoplus_{M \in {\mathcal X}^{(2)}} \Z/n$.

On a  $\Br {\mathcal X} =0$ si $k$ est s\'eparablement clos (voir \cite[Cor. 1.10]{CTOjPa})
ou si $k$ est fini (voir  \cite[Cor. 1.11]{CTOjPa}).
Ceci assure alors  la nullit\'e de l'homologie du complexe en le terme  $ {}_{n} \Br K$.

Pour $X_{0}/k$ une courbe  sur un corps s\'eparablement clos, on a
  $H^3_{\et}(X_{0},\mu_{n}^{\otimes 2}) =0$ et, si $X_{0}$ est propre et connexe,
   $H^2_{\et}(X_{0},\mu_{n}^{\otimes 2}) =\mu_{n}$.

Pour $X_{0}/k$ une courbe propre g\'eom\'etriquement connexe sur un corps fini $k$ de cl\^oture alg\'ebrique $k_{s}$,
on a  $$H^3_{\et}(X_{0},\mu_{n}^{\otimes 2}) \oi H^1(k,H^2(X_{0}\times_{k}k_{s},\mu_{n}^{\otimes 2}))
= H^1(k,\mu_{n}) = k^{\times}/k^{\times n} \simeq \Z/n,$$
le dernier isomorphisme provenant de l'hypoth\`ese $\Z/n \oi \mu_{n}$.

 Le th\'eor\`eme de changement de base propre donne des isomorphismes
 $H^{r}_{\et}({\mathcal X},\mu_{n}^{\otimes i}) \simeq H^{r}_{\et}(X_{0},\mu_{n}^{\otimes i}) $.
 On a donc  $H^{3}_{\et}({\mathcal X},\mu_{n}^{\otimes 2})=0$ si $k$ est s\'eparablement clos
 et $H^{3}_{\et}({\mathcal X},\mu_{n}^{\otimes 2}) \simeq \Z/n$ si $k$ est fini.

  Sous la conjecture de Gersten pour la cohomologie \'etale \`a coefficients $\mu_{n}$ 
sur le sch\'ema $\mathcal X$, on a   une suite exacte 
\begin{equation} \label{(modulogersten)}
0 \to H^1(C_{2,2}) \to H^3({\mathcal X}, \mu_{n}^{\otimes 2})  \to H^{0}(C_{3,2}) \to H^2(C_{2,2})
\end{equation}\label{sousgersten}
(on note ici $H^{r}(C_{j,i})$ le $r$-i\`eme groupe d'homologie du complexe de groupes ab\'eliens
$C_{j,i}$).  On a en particulier une inclusion $ H^1(C_{2,2}) \hookrightarrow H^3({\mathcal X}, \mu_{n}^{\otimes 2})$,
ce qui ach\`eve la d\'emonstration. 
\end{proof}

\begin{rema}
Si $k$ est un corps fini, il est  vraisemblable que le groupe
$$H^{0}(C_{3,2})  = Ker [H^3(K, \mu_{n}^{\otimes 2}) \to \bigoplus_{x \in {\mathcal X}^{(1)}} H^2(\kappa(x),\mu_{n})]$$
est nul.  Dans le cas (b), 
voir \cite[Prop. 3.8]{CTOjPa} et \cite[Prop. 4.1]{Hu2}.
Dans le cas (a), et pour $R$ complet, c'est un r\'esultat de Kato \cite[Thm. 5.2]{kato}, voir aussi
\cite[Thm. 3.3.6]{HHK3}.
Pour $k$ fini, le groupe de cohomologie m\'edian dans la proposition est alors \'egal \`a $\Z/n$.
\end{rema}

  \begin{rema}\label{complementBO}
  Bloch et Ogus \cite{BO} ont \'etabli la conjecture de Gersten pour  la cohomologie \'etale pour
 les vari\'et\'es lisses sur un corps.
 Le r\'esultat a \'et\'e \'etendu par I. A. Panin \cite{panin} en 2003 aux sch\'emas r\'eguliers contenant un corps,
en utilisant un th\'eor\`eme de Popescu.  

Pour $R$ une $k$-alg\`ebre de corps r\'esiduel $k$ s\'eparablement clos, et $n$ entier non nul dans $k$
et $\mathcal X$ comme ci-dessus, on a donc une suite exacte :
$$0 \to {}_{n} \Br K \to \bigoplus_{\gamma \in {\mathcal X}^{(1)}} \kappa(\gamma)^{\times}/\kappa(\gamma)^{\times n}
\to \bigoplus_{M \in {\mathcal X}^{(2)}} \Z/n \to 0.$$

Cette suite exacte joue un r\^ole-cl\'e au \S \ref{paradescente} ci-dessous.

Elle est analogue \`a la suite fondamentale de la th\'eorie du corps de classes sur un corps
global $K$ (corps de nombres ou corps de fonctions d'une variable sur un corps fini)
$$0 \to \Br K \to \oplus_{v \in \Omega_{K}} \Br K_{v} \to \Q/\Z \to 0.$$

  \end{rema}
  
  \begin{rema}\label{foncpropre}
  Pla\c cons-nous dans le cas \og semi-global \fg. Supposons
  donn\'e un isomorphisme $\Z/n \oi \mu_{n}$
  sur $R$.
 Soit $F$ le corps des fractions de l'anneau de valuation discr\`ete $R$.
    La fonctorialit\'e covariante par morphisme propre du complexe de Bloch-Ogus
    pour le morphisme ${\mathcal X} \to \Spec R$
  donne un diagramme commutatif de complexes
   \[\xymatrix{
 {}_{n}\Br K    \ar[d]  \ar[r]   & 
  \bigoplus_{\gamma \in {\mathcal X}^1} \kappa(\gamma)^{\times}/\kappa(\gamma)^{\times n}  
   \ar[d]     \ar[r]    & 
   \bigoplus_{x\in {\mathcal X}^{(2)} } \Z/n    \ar[d]     \\
0 \ar[r]  
 &  F^{\times}/F^{\times n}  \ar[r]  &   \Z/n
}\]
qu'il est d'ailleurs facile d'\'etablir directement.
Les fl\`eches verticales m\'edianes sont nulles sur
les composantes de $X_{0}$, et sont induites par la norme sur les corps de fonctions des
autres courbes.  Les fl\`eches verticales de droite sont induites
par le degr\'e relatif au corps r\'esiduel $k$.
La commutativit\'e de ce diagramme est facile \`a \'etablir.
Le seul point non trivial est la commutativit\'e du
carr\'e de droite, qui refl\`ete la loi de r\'eciprocit\'e de Weil
sur la fibre g\'en\'erique de ${\mathcal X}/R$.

Le noyau de la surjection $F^{\times}/F^{\times n} \to  \Z/n$
est $R^{\times}/R^{\times n} \oi k^{\times}/k^{\times n}$.
Pour $k$ s\'eparablement clos, ce noyau est donc trivial; 
pour $k$ fini, il est isomorphe \`a $\Z/n$.

 Il est vraisemblable
que l'application depuis  l'homologie m\'ediane du complexe sup\'erieur
dans $k^{\times}/k^{\times n}$ s'identifie \`a la fl\`eche compos\'ee
$$H^1(C_{2,2}) \to H^3_{\et}({\mathcal X}, \mu_{n}^{\otimes 2})  \oi
H^3_{\et}(X_{0},\mu_{n}^{\otimes 2}) \to H^1(k,H^2(X_{0}\times_{k}k_{s},\mu_{n}^{\otimes 2})),$$
o\`u la premi\`ere fl\`eche est comme dans (\ref{sousgersten}), la seconde est obtenue par le th\'eor\`eme
de changement de base propre, et la troisi\`eme vient de la suite spectrale de Leray
pour la $k$-courbe $X_{0}$.

  \end{rema}
  
  \subsection{Obstruction de r\'eciprocit\'e au-dessus d'une courbe}\label{recipweil}
  
 Pour mettre les obstructions du paragraphe suivant en perspective,
 nous rappelons dans cette sous-section une obstruction introduite
  dans un cas particulier dans \cite{CTreel} 
  par analogie avec l'obstruction d\'efinie par Manin  \cite{Manin}
  sur un corps de nombres.
  Avec les notations ci-dessous, le cas particulier
  o\`u $r=2$, $i=1$ et $F=\F$ est un corps  fini
  correspond \`a l'obstruction de Brauer-Manin sur le corps global
  $\F(X)$, du moins pour la torsion premi\`ere \`a la caract\'eristique.
   En degr\'e cohomologique sup\'erieur, sur un corps de fonctions d'une variable
    sur $F$  le corps   des r\'eels, resp.  sur $F$  un corps $p$-adique,
  cette obstruction a \'et\'e utilis\'ee  par Ducros
  \cite{Ducros}, resp.   par Harari et Szamuely \cite{HaSz}.  
 Ils  ont montr\'e que pour certaines classes de vari\'et\'es, les obstructions ainsi
  d\'efinies \`a l'existence de points rationnels sont les seules.

 \medskip
  
Soit $F$ un corps. Soit $X$ une $F$-courbe g\'eom\'etriquement int\`egre, projective et lisse.
 Soit $K=F(X)$. Soit $n$ un entier premier \`a la caract\'eristique.
 Pour $x \in X^{(1)}$ on note $K_{x}$ le corps des fractions du hens\'elis\'e $O^h_{X,x}$ de l'anneau local $O_{X,x}$   et $F(x)$ le corps r\'esiduel.
 
 Soit $i \in \Z$.  Nous prenons ici  $\mu_{n}^{\otimes i}$ comme coefficients. On pourrait prendre
 tout module galoisien fini sur $F$ d'ordre premier \`a la caract\'eristique de $F$.
 
 On a la longue suite exacte   de localisation en cohomologie \'etale :
 $$H^r(X, \mu_{n}^{\otimes i})  \to  H^{r}(F(X), \mu_{n}^{\otimes i}) \to \oplus_{x \in X^{(1)}} 
  H^{r-1} (F(x), \mu_{n}^{\otimes (i-1)})
  \to    H^{r+1}(X, \mu_{n}^{\otimes i}   ) \to  $$
  pour laquelle nous renvoyons au rapport \cite[\S 3]{CTBarbara}.

 L'application  $\partial_{x} : H^{r}(F(X), \mu_{n}^{\otimes i})  \to H^{r-1} (F(x), \mu_{n}^{\otimes ( i-1)})$ est le r\'esidu en $x$
 (voir \cite[\S 6.8]{GiSz}).
 
On dispose des applications de corestriction
$$\Cores_{F(x)/F} : H^{r-1} (F(x), \mu_{n}^{\otimes (i-1)}) \to H^{r-1}(F,\mu_{n}^{\otimes (i-1)}).$$
 On a la loi de r\'eciprocit\'e de Weil g\'en\'eralis\'ee : la suite suivante  est un complexe
 $$H^{r}(F(X), \mu_{n}^{\otimes i})
  \to \oplus_{x \in X^{(1)}  }H^{r-1} (F(x), \mu_{n}^{\otimes  (i-1)}) 
  \to H^{r-1}(F,\mu_{n}^{\otimes  (i-1)}).
 $$
Voir \`a ce sujet
 \cite[Annexe, \S 3]{serreCG},  
  \cite[\S 6.9]{GiSz}.

 Soit $Z$ une $K$-vari\'et\'e int\`egre projective et lisse. Nous renvoyons \`a \cite[\S 4]{CTBarbara} pour
 la d\'efinition et les propri\'et\'es de base de la cohomologie non ramifi\'ee.
 
  D'apr\`es la  conjecture de Gersten pour la cohomologie \'etale,
 \'etablie par Bloch et Ogus  \cite{BO},
 en tout point  $P $ de la $K$-vari\'et\'e lisse $Z$ 
 une classe dans  $H^{r}_{nr}(K(Z)/K, \mu_{n}^{\otimes i})$
 est l'image d'un (unique) \'el\'ement de $H^r(O_{Z,P},\mu_{n}^{\otimes i})$.
Pour tout corps $L$ contenant $K$, ceci d\'efinit un accouplement
$$H^{r}_{nr}(K(Z)/K, \mu_{n}^{\otimes i})  \times Z(L) \to H^r(L,\mu_{n}^{\otimes i})$$
$$ (\alpha,P) \mapsto \alpha(P).$$
Dans la situation ci-dessus, pour chaque $x \in X^{(1)}$, en utilisant le r\'esidu en $x$,
on obtient un accouplement
 $$H^{r}_{nr}(K(Z)/K, \mu_{n}^{\otimes i})  \times Z(K_{x}) \to  H^{r-1} (F(x), \mu_{n}^{\otimes  (i-1)})$$
 $$ (\alpha, P_{x}) \mapsto \partial_{x}(\alpha(P_{x})).$$

 \begin{prop}\label{weilrecip}
(i)  L'accouplement
 $$ H^{r}_{nr}(K(Z)/K, \mu_{n}^{\otimes i}) \times \prod_{x \in X^{(1)}} Z(K_{x}) \to H^{r-1}(F,\mu_{n}^{\otimes  (i-1)}) $$
  $$ (\alpha, \{P_{x}\}) \mapsto  \sum_{x \in X^{(1)}} \Cores_{F(x)/F}(\partial_{x}(\alpha(P_{x})))$$
  est bien d\'efini.

(ii) L'accouplement est nul lorsque $\{P_{x}\}$ est dans 
  l'image diagonale de $Z(K)$ dans $  \prod_{x \in X^{(1)}} Z(K_{x})$.
 \end{prop}
 
\begin{proof}
  Pour  \'etablir (i), il suffit de montrer que
pour presque tout $x \in X^{(1)}$, l'application 
$$Z(K_{x}) \to H^{r-1} (F(x), \mu_{n}^{\otimes  (i-1)})$$
$$ P \mapsto \partial_{x}(\alpha(P))$$
d\'efinie par $\alpha$ est nulle.
Il existe un ouvert non vide $U \subset X$ et
 ${\mathcal Z}  \to U$ un mod\`ele projectif et lisse de $Z \to \Spec K$.
 Une classe dans $H^{r}(F({\mathcal Z}), \mu_{n}^{\otimes i})$ n'a qu'un nombre fini
 de r\'esidus non nuls sur la $F$-vari\'et\'e lisse ${\mathcal Z}$.
 Comme par hypoth\`ese $\alpha$ a tous ses r\'esidus nuls sur $Z$,
 quitte \`a restreindre $U$ on peut supposer que $\alpha$ a tous ses
 r\'esidus nuls sur $\mathcal Z$.
 Soit alors $x \in U$ et $P \in Z(K_{x})$. 
 On a $Z(K_{x})={\mathcal Z}(O^h_{X,x})$.  Le point $P$
 d\'efinit donc un morphisme de $\Spec O^h_{X,x}$ dans $\mathcal Z$.
 Soit $N \in  \mathcal Z$ l'image du point ferm\'e de $O^h_{X,x}$.
 D'apr\`es Bloch-Ogus (la conjecture de Gersten) appliqu\'e \`a la $k$-vari\'et\'e lisse $\mathcal Z$,
 la classe $\alpha$ est image d'un \'el\'ement bien d\'efini de l'anneau
 local $H^{r}(O_{{\mathcal Z},N},   \mu_{n}^{\otimes i})$.
 On voit alors que $\alpha(P) $ appartient \`a $H^{r}(O^h_{X,x}, \mu_{n}^{\otimes i})$,
 et a donc un r\'esidu nul dans $H^{r-1} (F(x), \mu_{n}^{\otimes  (i-1)}) $.
 Ceci montre que l'accouplement
 $$ H^{r}_{nr}(K(Z)/K, \mu_{n}^{\otimes i}) \times \prod_{x \in X^{(1)}} Z(K_{x}) \to H^{r-1}(F,\mu_{n}^{\otimes  (i-1)}) $$
 $$ (\alpha, \{P_{x}\} ) \mapsto  \sum_{x \in X^{(1)}}\Cores_{F(x)/F}(\partial_{x}(\alpha(P_{x})))$$
est bien d\'efini. Ceci \'etablit le point (i). Le point (ii) est une cons\'equence
de la loi de r\'eciprocit\'e de Weil sur la courbe $X$.
\end{proof}

 \begin{ex}
On ne saurait esp\'erer, m\^eme pour les vari\'et\'es les plus simples,
que cette obstruction cohomologique suffise \`a d\'ecider de l'existence d'un point.
L'exemple suivant est une variante simple de celui donn\'e dans \cite[Rem. 7.9]{CTPaSu}.
Soit $X/\C((t))$ la courbe elliptique donn\'ee par l'\'equation affine
$$\rho^2=\sigma(\sigma+1)(\sigma-t).$$
Soit  $K$ son corps des fonctions.
Soit $Z/K$ la conique d'\'equation projective
$$U^2-\sigma V^2-tW^2=0.$$
Au voisinage de tout point  ferm\'e $x \in X$, on peut \'ecrire $\sigma \in K$ comme le produit d'une unit\'e en $x$ et
d'un carr\'e. Comme le corps r\'esiduel de tout tel point est un corps $C_{1}$,
ceci implique $Z(K_{x}) \neq \emptyset$ pour tout tel point $x$.
Comme $Z$ est une conique, les applications $H^{i}(K,\Z/n) \to H^{i}_{nr}(K(Z)/K, \Z/n)$
sont toutes surjectives. Il n'y a donc pas d'obstruction de r\'eciprocit\'e donn\'ee par la proposition
\ref{weilrecip}.

La courbe elliptique $X$ admet une r\'eduction modulo $t=0$ d'\'equation affine
$\rho^2=\sigma^2(\sigma+1)$ sur $\C$, qui est birationnelle \`a $\tau^2=\sigma+1$ et donc \`a $\Spec \C[\tau]$
 via $\rho=\tau.\sigma$.
La conique $Z$  n'a pas de point dans le hens\'elis\'e $K_{t}$
de $K$ le long de $t=0$, car le r\'esidu de l'alg\`ebre de quaternions
$(\sigma,t)$ est donn\'e par $\sigma=\tau^2-1 \in \C(\tau)^{\times}/\C(\tau)^{\times 2}$.

Ceci implique $Z(K)=\emptyset$. 
Comme $Z$ est une $K$-compactification lisse 
d'un espace principal homog\`ene sous un $K$-tore,
ceci donne aussi un exemple du m\^eme type pour de telles vari\'et\'es.
Cela donne aussi un exemple o\`u l'application
$\Br K \to \prod_{x\in X^{(1)} } \Br K_{x}$
n'est pas injective.
 \end{ex}

\subsection{Au-dessus d'une base
de dimension quelconque : 
de nouvelles obstructions de r\'eciprocit\'e}

Les obstructions suivantes \`a l'existence de points rationnels 
 ont  \'et\'e propos\'ees il y a quelques ann\'ees
par Colliot-Th\'el\`ene.
Faute d'application, elles n'avaient pas \'et\'e publi\'ees. Le pr\'esent article les met pour
la premi\`ere fois en {\oe}uvre au-dessus d'une base de dimension plus grande que 1.

\begin{prop}\label{obstrucsup}
Soit ${\mathcal X}$ un sch\'ema r\'egulier excellent int\`egre de dimension $d$, tel que
pour tout ferm\'e irr\'eductible ${\mathcal F} \subset  {\mathcal X}$,
on a la propri\'et\'e ${\dim}({\mathcal F}) = d - {\rm codim}_{\mathcal X}({\mathcal F})$.
Soit $K$ le  corps   de
   fonctions de $\mathcal X$. Soit $n \in \N$
un entier inversible sur ${\mathcal X}$ et soient $r\in \N$ et $i\in \Z$.
Soit $Z$ une $K$-vari\'et\'e projective, lisse, g\'eom\'etriquement int\`egre.

 (i) Pour tout point $\gamma$ de codimension 1 de ${\mathcal X}$, on dispose d'une \'evaluation
  $$ H^{r}_{nr}(K(Z)/K, \mu_{n}^{\otimes i})  \times   Z(K_{\gamma})    
  \to H^{r}(K_{\gamma}, \mu_{n}^{\otimes i}) $$
  $$(\alpha,P) \to \alpha(P)$$
qui par application du r\'esidu $\partial_{\gamma}$ induit un accouplement
   $$H^{r}_{nr}(K(Z)/K, \mu_{n}^{\otimes i})  \times   Z(K_{\gamma})  \to  
  H^{r-1}(\kappa(\gamma),\mu_{n}^{\otimes ( i-1)}).$$
  
  (ii) Supposons que ${\mathcal X}$ est un sch\'ema au-dessus d'un corps.
   Pour tout $\alpha \in  H^{r}_{nr}(K(Z)/K, \mu_{n}^{\otimes i})$,
  il existe un ouvert $U$ d\'ependant 
  de $\alpha$ tel que
pour tout $\gamma \in {\mathcal X}^{(1)}$ appartenant \`a $U$  l'application
  $$Z(K_{\gamma}) \to H^{r-1}(\kappa(\gamma),\mu_{n}^{\otimes ( i-1)})$$
  d\'efinie par $\alpha$ a son image nulle.
  
  (iii) Supposons que ${\mathcal X}$ est un sch\'ema au-dessus d'un corps.  On dispose d'un accouplement
  bien d\'efini
  $$ H^{r}_{nr}(K(Z)/K, \mu_{n}^{\otimes i}) \times \prod_{  \gamma \in {\mathcal X}^{(1)}   }     Z(K_{\gamma}) \to
  \bigoplus_{M \in {\mathcal X}^{(2)}  }H^{r-2}(\kappa(M),\mu_{n}^{\otimes ( i-2)})$$
  $$ (\alpha, \{P_{\gamma}\} ) \mapsto  \sum_{\gamma} \partial_{\gamma,M}(\partial_{\gamma}(\alpha(P_{\gamma}))) $$
 et l'accouplement est nul lorsque $\{P_{\gamma}\}$ est dans 
  l'image diagonale de $Z(K)$ dans $  \prod_{\gamma \in {\mathcal X}^{(1)}} Z(K_{\gamma})$.

  (iv) Si  $r=1$, ou si $r=2$ et $ i=1$, les \'enonc\'es valent sans supposer que ${\mathcal X}$ est un sch\'ema au-dessus d'un corps.
  \end{prop}
  
\begin{proof}
   L'\'enonc\'e (i) r\'esulte de la th\'eorie de Bloch-Ogus pour la vari\'et\'e lisse $Z$ sur le corps $K$.
  
    Il existe un ouvert non vide $U \subset {\mathcal X}$ et un
  mod\`ele projectif et lisse ${\mathcal Z} \to U$ de $Z/K$. Pour
  $\alpha \in H^{r}_{nr}(K(Z)/K, \mu_{n}^{\otimes i}) $ donn\'e, apr\`es restriction de $U$
on peut assurer que  $\alpha$ a tous ses r\'esidus trivaux sur $\mathcal Z$ .
Soit $P_{\gamma} \in Z(K_{\gamma})$. 
Comme on a  $Z(K_{\gamma})={\mathcal Z}(O^h_{Z,\gamma})$, le point $P_{\gamma} $
s'\'etend en un morphisme $\Spec O^h_{Z,\gamma}  \to {\mathcal Z}$. Soit $N$
l'image du point ferm\'e de  $\Spec  O^h_{Z,\gamma}$. Par la th\'eorie de Bloch-Ogus 
sur le sch\'ema r\'egulier ${\mathcal Z}$ (on utilise ici \cite{panin}),
la classe
$\alpha$ provient d'une classe dans $H^{r}(O_{{\mathcal Z},N},\mu_{n}^{\otimes i})$.
On a donc $\alpha(P_{\gamma}) \in H^{r}( O^h_{Z,\gamma},\mu_{n}^{\otimes i})$ et donc
$\partial_{\gamma}(\alpha(P_{\gamma}))=0$. Ceci \'etablit (ii) et donc le fait que l'accouplement en (iii)
est bien d\'efini.

Par la conjecture de Gersten \cite{BO}  pour la $K$-vari\'et\'e $Z$, 
 pour tout $\alpha \in H^{r}_{nr}(K(Z)/K, \mu_{n}^{\otimes i})$
et tout $P \in Z(K)$, la classe $\alpha$ est repr\'esent\'ee par un unique  \'el\'ement
de $H^{r}(O_{Z,P},  \mu_{n}^{\otimes i})$, qui d\'efinit une classe $\alpha(P) \in H^{r}(K, \mu_{n}^{\otimes i})$,
dont l'image dans   $ H^{r}(K_{\gamma}, \mu_{n}^{\otimes i})$
co\"{\i}ncide avec sa valeur calcul\'ee sur $K_{\gamma}$. La fin de l'\'enonc\'e (iii) r\'esulte alors
du fait que les applications $\partial_{\gamma}$ et    $\partial_{\gamma,M}$ d\'efinissent un complexe
$$  H^{r}(K, \mu_{n}^{\otimes i}) \to \bigoplus_{\gamma \in {\mathcal X}^{(1)} }
 H^{r-1}(\kappa(\gamma), \mu_{n}^{\otimes (i-1)}) \to
\bigoplus_{M  \in {\mathcal X}^{(2)} } H^{r-2}(\kappa(M),  \mu_{n}^{\otimes (i-2)})   $$
 comme rappel\'e
au d\'ebut du paragraphe \ref{BOK}.

L'\'enonc\'e (iv) r\'esulte du fait que dans ces deux cas,
pour tout sch\'ema r\'egulier  $\mathcal Z$
  de corps de
  fonctions $F$,
   si une classe dans $H^{r}(F,\mu_{n}^{\otimes i})$
  a tous ses r\'esidus triviaux sur $\mathcal Z$, alors elle provient de $H^{r}_{\et}({\mathcal Z},\mu_{n}^{\otimes i})$.
  Le cas $r=1$ est classique, le cas $r=2$ est une cons\'equence  du th\'eor\`eme de puret\'e absolue de Gabber \cite[Thm. 3.1.1]{surGab}. 
\end{proof}

\begin{rema}  Sous les hypoth\`eses de la proposition \ref{obstrucsup},  on voit donc qu'une condition n\'ecessaire
d'existence d'un $K$-point sur $Z$ est l'existence d'une famille $\{P_{\gamma}\} \in \prod_{\gamma  \in {\mathcal X}^{(1)}}Z(K_{\gamma})$
telle que, pour tout $i$, tout  $r$, tout $n$, tout $\alpha \in H^{r}_{nr}(K(Z)/K, \mu_{n}^{\otimes i})$,
et tout $M \in {\mathcal X}^{(2)}$, on ait
$$ \sum_{\gamma} \partial_{\gamma,M}(\partial_{\gamma}(\alpha(P_{\gamma})))  = 0 \in   H^{r-2}(\kappa(M),  \mu_{n}^{\otimes (i-2)}),$$
o\`u la somme porte sur les $\gamma  \in {\mathcal X}^{(1)}$ dont l'adh\'erence contient $M$.

\end{rema}

\begin{rema}
Soit ${\mathcal Y}  \to  {\mathcal X}$ un morphisme propre birationnel de 
sch\'emas excellents r\'eguliers
int\`egres satisfaisant l'hypoth\`ese sur les dimensions faite au d\'ebut de la proposition \ref{obstrucsup}.
Soit $K$ le corps de fonctions 
de ${\mathcal X}$.
Soit $Z$ une $K$-vari\'et\'e projective, lisse, g\'eom\'etriquement int\`egre.
Soit $\alpha \in  H^{r}_{nr}(K(Z)/K, \mu_{n}^{\otimes i})$. Sous les hypoth\`eses de la proposition \ref{obstrucsup},
s'il existe une famille $\{P_{\gamma}\}_{\gamma \in {\mathcal Y}^{(1)}}$ telle que pour tout $M \in {\mathcal Y}^{(2)}$
on ait
$$\sum_{\gamma} \partial_{\gamma,M}(\partial_{\gamma}(\alpha(P_{\gamma}))) =0, $$
alors la famille  $\{Q_{\zeta}\}_{\zeta \in {\mathcal X}^{(1)}}$  d\'efinie par $Q_{\zeta}=P_{\gamma}$,
o\`u $\gamma \in {\mathcal Y}$ d\'esigne   l'unique point de codimension 1 de $\mathcal Y$ au-dessus de $\zeta$
satisfait, pour tout $N \in {\mathcal X}^{(2)}$
$$\sum_{\zeta} \partial_{\zeta,N}(\partial_{\zeta}(\alpha(Q_{\zeta}))) =0 \in H^{r-2}(\kappa(N),  \mu_{n}^{\otimes (i-2)}).$$
Ceci r\'esulte de la  la fonctorialit\'e covariante par morphisme propre 
du complexe de Bloch-Ogus \cite[\S 2]{JaSa}.

On peut se demander s'il existe un \'enonc\'e permettant de remonter de propri\'et\'es
sur $\mathcal X$ \`a des propri\'et\'es sur $\mathcal Y$. On en verra un exemple
  tr\`es particulier 
\`a la section \ref{eclatement}.

\end{rema}

  \begin{prop}\label{descentealgebre}
  Soient  $k$ un corps s\'eparablement clos, $R$ une $k$-alg\`ebre
  locale int\`egre,  hens\'elienne, excellente, 
de corps r\'esiduel $k$,  de corps des fractions $F$,
 puis $ {\mathcal X}$ un sch\'ema r\'egulier int\`egre de dimension 2
 \'equip\'e d'un morphisme projectif
  $ p : {\mathcal X} \to \Spec R$.
  Supposons que l'on est dans l'un des cas suivants.

(a) L'anneau $R$ est un anneau de valuation discr\`ete,
 les fibres de $p$
sont de dimension 1, la fibre g\'en\'erique est lisse et g\'eom\'etriquement int\`egre.

(b) L'anneau  $R$ est  de dimension 2, et $p$ est birationnel.

 Soit $K$ le corps des fonctions rationnelles de $\mathcal X$. 
 Soit $n>0$ un entier inversible dans $k$.
  Soit  $Z$ une $K$-vari\'et\'e projective, lisse, g\'eom\'etriquement int\`egre,
 et soit $A \in {}_{n}\Br Z$.  Fixons un isomorphisme $\Z/n \oi \mu_{n}$.
 
Supposons qu'il existe une famille $\{P_{\gamma}\}_{\gamma \in {\mathcal X}^{(1)}}$,
 telle que la famille  $\{ \partial_{\gamma}(A(P_{\gamma})) \}_{\gamma \in {\mathcal X}^{(1)}}$ soit dans le noyau de la fl\`eche
 $ \oplus_{\gamma \in {\mathcal X}^{(1)}} k(\gamma)^{\times}/k(\gamma)^{\times n}
\to \oplus_{M \in {\mathcal X}^{(2)}} \Z/n. $
Alors :

(i) Il existe $\alpha \in \Br K$ tel que pour tout $\gamma$ on ait
$A(P_{\gamma})=\alpha \in \Br K_{\gamma}$, et cet \'el\'ement $\alpha$ est uniquement d\'etermin\'e.

(ii) Dans le cas (a), soit  $X={\mathcal X}\times_{R}F$. Alors
  $\{ \partial_{\gamma}(A(P_{\gamma})) \}_{\gamma \in X^{(1)}}$ est dans le noyau de la fl\`eche
   $$ \sum_{\gamma \in X^{(1)}} 
   \Cores_{F(\gamma)/k} : 
   \bigoplus_{   \gamma \in X^{(1)}      }   
   H^1(F(\gamma),\mu_{n}) 
  \to H^1(F,\mu_{n}    ).
    $$
 \end{prop}

\begin{proof}
L'\'enonc\'e (i)  r\'esulte de la proposition  \ref{blochogus} 
et de la remarque \ref{complementBO}
  et du fait que, comme $k$
est s\'eparablement clos, la fl\`eche de r\'esidu
$$ {}_{n}\Br K_{\gamma}  \to k(\gamma)^{\times}/k(\gamma)^{\times n}$$
est un isomorphisme. L'\'enonc\'e (ii)  r\'esulte de  la remarque
\ref{foncpropre} et du fait que la fl\`eche 
$$F^{\times}/F^{\times n} \to \Z/n$$
induite par la valuation est un isomorphisme.
\end{proof}
 
 \begin{rema}
 Pour $k$ un corps fini plut\^ot que s\'eparablement clos, l'\'enonc\'e ci-dessus
 ne vaut pas. On a  une obstruction possible  :
 la famille  $\{ \partial_{\gamma}(A(P_{\gamma})) \}_{{\gamma \in {\mathcal X}^{(1)} }}$ pourrait \^etre un \'el\'ement non nul
 dans l'homologie du complexe de Bloch-Ogus, qui est  alors un sous-groupe de $\Z/n$
 (Prop. \ref{blochogus}).

 Pour $\kappa$ un corps fini il serait  int\'eressant d'utiliser les obstructions
 \`a l'existence d'un $K$-point sur $Z$ d\'efinies par le groupe $H^3_{nr}(K(Z)/K,\mu_{n}^{\otimes 2})$.
 Voir \`a ce sujet \cite{HaSz}.
 \end{rema}

\section{L'\'equation $(X_{1}^2-aY_{1}^2)(X_{2}^2-bY_{2}^2)(X_{3}^2-abY_{3}^2)=c$}\label{algebre}

\subsection{Rappels}

 Soit $K$ un corps de caract\'eristique diff\'erente de 2, soient
 $a,b,c \in K^{\times}$. 
 Dans \cite{CT12}, on a \'etudi\'e le $K$-tore $Q$
    d\'efini par l'\'equation
 $$(X_{1}^2-aY_{1}^2)(X_{2}^2-bY_{2}^2)(X_{3}^2-abY_{3}^2)=1.$$
 Pour $c\in K^{\times}$, on note $E=E_{c}$
l'espace principal homog\`ene sous $Q$ d\'efini par l'\'equation
  $$(X_{1}^2-aY_{1}^2)(X_{2}^2-bY_{2}^2)(X_{3}^2-abY_{3}^2)=c,$$
  avec $c \in K^{\times}$.
Tout espace principal homog\`ene sous $Q$ s'\'ecrit ainsi pour un
$c$ convenable.

   Soit  $L$ la $K$-alg\`ebre \'etale $K[u,v]/(u^2-a, v^2-b)$. 
   Pour $\rho \in K^{\times}$, soit $K_{\rho}$ la $K$-alg\`ebre \'etale $K[x]/(x^2-\rho)$. 
   On d\'efinit des inclusions naturelles et compatibles $K_{a} \subset L$, $K_{b} \subset L$,
   $K_{ab} \subset L$.
   Soit $d \in K^{\times}$.
  Soit $E'_{d}$ la $K$-vari\'et\'e
 d'\'equation
 $$\Norm_{L/K}(\Xi) =  d.$$
 La $K$-vari\'et\'e $E'_{d}$ est un espace principal homog\`ene
sour le $K$-tore $T$ d'\'equation
$$\Norm_{L/K}(\Xi) = 1,$$
et tout espace principal homog\`ene sous $T$ est de la forme $E'_{d}$
 pour $d$ convenable.

Rappelons qu'un espace principal homog\`ene $E$  d'un $K$-tore $M$
a un point rationnel sur $K$ si et seulement si la classe $[E] \in H^1(K,M)$
est nulle.

D'apr\`es    \cite[Prop. 3.1]{CT12},
on a une suite exacte de $K$-tores
$$ 1 \to \G_{m,K}^2 \to Q \to T \to 1,$$
o\`u
l'application  $Q \to T$ est induite par l'application
$K_{a}^{\times} \times K_{b}^{\times}  \times K_{ab}^{\times} \to L^{\times}$
envoyant  un triplet $(\alpha_{a},\alpha_{b},\alpha_{ab})$ sur le produit 
$\alpha_{a}.\alpha_{b}.\alpha_{ab}$.
L'application $H^1(K,Q) \to H^1(K,T)$ envoie la classe
de $E=E_{c}$ sur la classe de $E'_{c^2}$.
 
Le th\'eor\`eme 90 de Hilbert donne $H^1(K,\G_{m})=0$. 
La suite exacte ci-dessus induit donc une injection
fonctorielle en le corps de base
$$ H^1(K,Q) \hookrightarrow H^1(K,T).$$  

\medskip

\begin{lem}
Soient $K$ un corps et $M$ un $K$-tore. Il existe 
un $K$-groupe de type multiplicatif fini $\mu$ 
et une inclusion fonctorielle en le corps de base
$$H^1(K,M) \hookrightarrow H^2(K,\mu).$$
\end{lem}

\begin{proof} Soit $L/K$ une extension finie galoisienne,
de groupe de Galois $G$, d\'eployant le $K$-tore $M$.
Soit $n=[L:K]$ son degr\'e. Soit $\mu={}_{n}M$ le noyau de
la multiplication par $n$.
La multiplication par $n$ sur $M$ induit une suite exacte
de
  $K$-groupes de type multiplicatif
$$1 \to \mu \to  M \to M \to 1.$$
Le th\'eor\`eme 90 de Hilbert et un argument de corestriction-restriction
montre que le groupe $H^1(K,M)$ est annul\'e par $n$.
La suite exacte de cohomologie donne donc naissance \`a un plongement
$H^1(K,M) \hookrightarrow H^2(K,\mu)$, qui est fonctoriel en le corps de base.
\end{proof}
 
 Appliquant le lemme au cas $M=T$,  on obtient  une inclusion $ H^1(K,T) \hookrightarrow H^2(K,\mu)$ 
 fonctorielle en $K$, et donc aussi une inclusion $ H^1(K,Q) \hookrightarrow H^2(K,\mu)$
 fonctorielle en $K$.

  \bigskip
  
  Soit $Z$ une $K$-compactification lisse de $E=E_{c}$.
   On sait  \cite{CTHaSk} que sur tout corps il existe une telle compactification, qu'on peut m\^eme choisir
  \'equivariante pour l'action de $Q$ sur lui-m\^eme.
  On sait  \cite[Lemma 2.1 (iv)]{BoCTSk}
  que l'on a $Z(K)\neq \emptyset$ si et seulement si $E(K)\neq \emptyset$.
 
  On a montr\'e dans  \cite{CT12} que
   le groupe de Brauer de $Z$ modulo le groupe
  de Brauer de $K$ est engendr\'e par la classe de l'alg\`ebre de quaternions
  $$A= (X_{1}^2-aY_{1}^2,b) \in   \Br K(E).$$
   Dans le pr\'esent article,
  il nous suffira d'utiliser le fait \'evident que la classe $A$ appartient 
  \`a $\Br E$.

   En utilisant la bilin\'earit\'e du symbole $(\alpha,\beta)$, la relation
   $(u^2-av^2,a)=0$ pour tout $u,v,a$ avec $a(u^2-av^2)\neq 0$, et l'\'equation 
     $$(X_{1}^2-aY_{1}^2)(X_{2}^2-bY_{2}^2)(X_{3}^2-abY_{3}^2)=c,$$
 on voit  que la  classe   $A= (X_{1}^2-aY_{1}^2,b) \in   \Br K(E)$ 
 peut s'\'ecrire de multiples fa\c cons :
  
  \begin{equation}\label{Avia1}
  A= (X_{1}^2-aY_{1}^2,b)= (X_{1}^2-aY_{1}^2,ab),
  \end{equation}
 
   \begin{equation}\label{Avia2}
 A= (X_{2}^2-bY_{2}^2, a) +  (c,ab) =    (X_{2}^2-bY_{2}^2, ab) +  (c,ab) = (c(X_{2}^2-bY_{2}^2) , ab),
    \end{equation}
  
    \begin{equation}\label{Avia3}
   A=    (X_{3}^2-abY_{3}^2, a)  + (c,b)=   (X_{3}^2-abY_{3}^2, b)  + (c,b)  =  (c(X_{3}^2-abY_{3}^2), b).
     \end{equation}

   \bigskip

 \subsection{Un exemple}

Soient $R$ un anneau de valuation discr\`ete hens\'elien,  $F$ son corps des fractions,
$k$ son corps r\'esiduel, suppos\'e parfait
de caract\'eristique diff\'erente de 2, $\pi$ une uniformisante de $R$.
On note $v$ la valuation de $F$.

Soient ${\mathcal X}= {\mathbf P}^1_{R}$ et $X=  {\mathbf P}^1_{F}$. Soit $K$ le corps des fonctions rationnelles sur~${\mathcal X}$.
Notons $\Spec R[x] = {\mathbf A}^1_{R} \subset   {\mathbf P}^1_{R}$.
Notons $\eta$ le point g\'en\'erique de la fibre sp\'eciale ${\mathbf P}^1_{k} \subset   {\mathbf P}^1_{R}$.

Soit $c \in F^{\times}$. La r\'eunion des supports des diviseurs de $x$, $ x+1$,  $c$
et de la fibre sp\'eciale est un diviseur \`a croisements normaux sur ${\mathcal X}=  {\mathbf P}^1_{R}$,
et ce diviseur est un arbre, union de $x=0$, $x+1=0$, $x=\infty$ et $\pi=0$.

Soit $E_{c}$ la $k$-vari\'et\'e d\'efinie par l'\'equation
$$ (X_{1}^2-xY_{1}^2)(X_{2}^2-(x+1)Y_{2}^2)(X_{3}^2-x(x+1)Y_{3}^2) = c.$$

\begin{prop}\label{pascontrex}
Avec les hypoth\`eses et notations ci-dessus, supposons que $k$ est  le corps
$\C$ des complexes, ou  que $k$ est  un corps fini $\F$ dans lequel $-1$ est un carr\'e.

\smallskip

(a) Pour tout point $\gamma \in X^{(1)}$, on a $E_{c}(K_{\gamma}) \neq \emptyset$.

   Les conditions suivantes sont \'equivalentes :

(b1) $E_{c}(K)=\emptyset$.

(b2) $E_{c}(K_{\eta})=\emptyset$.

(b3) $c \notin F^{\times 2}$.

(b4) La proposition  \ref{weilrecip} appliqu\'ee \`a  $A=(X_{1}^2-xY_{1}^2,x+1) \in 
H^2_{nr}(K(E_{c}),\Z/2)$
et \`a la $F$-courbe $X$ donne une obstruction de r\'eciprocit\'e \`a l'existence d'un point dans $E_{c}(K)$.

\end{prop}

\begin{proof}
Pour tout corps $\ell$  extension finie de $k$, le quotient
$\ell^{\times}/\ell^{\times 2}$ a au plus deux \'el\'ements et   $-1$
est un carr\'e dans $\ell$.
Un calcul simple
 montre qu'en tout point ferm\'e  $\gamma$ de la droite projective $ {\mathbf P}^1_{F}$,
l'un de $x, x+1, x(x+1)$ est un carr\'e dans le hens\'elis\'e $K_{x}$.
Pour tout $c \in F^{\times}$ et tout  point ferm\'e  $\gamma$ de la droite projective $ {\mathbf P}^1_{F}$,
on a alors clairement $E_{c}(K_{\gamma}) \neq \emptyset$.

Au point g\'en\'erique $\eta$  de la fibre sp\'eciale $ {\mathbf P}^1_{k} \subset  {\mathbf P}^1_{R}$,
chacun de $x$,  $x+1$, $x(x+1)$ est une unit\'e, mais aucun n'est un carr\'e dans le corps r\'esiduel
en $\eta$. Ceci implique que pour $X_{i}, Y_{i} \in K_{\eta}$ tels que le produit
$$(X_{1}^2-xY_{1}^2)(X_{2}^2-(x+1)Y_{2}^2)(X_{3}^2-x(x+1)Y_{3}^2) $$
soit non nul, la valuation en $\eta$ de ce produit est paire (voir le lemme \ref{facile} (ii)).

Si   $v(c)$ est impaire, on voit que $E_{c}(K_{\eta}) = \emptyset$. Ainsi $E_{c}(K) = \emptyset$.

Supposons  $v(c)$ paire. On peut supposer que  $c$ est une unit\'e dans $R$.
Si $k=\C$, alors   $c$ est un carr\'e dans $R$
et de fa\c con \'evidente $E_{c}(K) \neq \emptyset$.

Soit $A = (X_{1}^2-xY_{1}^2, x+1) \in {}_{2}\Br E_{c}$, qui est non ramifi\'ee sur un mod\`ele
projectif et lisse $Z_{c}$ de $E_{c}$ sur $K$. La proposition 5.1 (d) de \cite{CT12} permet de calculer
l'obstruction  associ\'ee \`a $A \in H^2_{nr}(K(E)/K,\Z/2)$  par la 
proposition \ref{weilrecip}. 
On trouve que pour toute famille $\{P_{\gamma} \in E(K_{\gamma})\}$, $\gamma$  parcourant les points ferm\'es
de $\P^1_{F}$, 
$$ \sum_{\gamma} \Cores_{F(\gamma)/F} \partial_{\gamma}(A(P_{\gamma})) = c \in F^{\times}/F^{\times 2}=H^1(F,\Z/2).$$
Ainsi $E(K)=\emptyset$ si $v(c)$ est impaire, mais aussi si $v(c)$ est paire
et $c$ non carr\'e, ce qui peut se produire si $k=\F$ est un corps fini.
\end{proof}

\begin{rema}
Soit $K$ le corps des fonctions d'une courbe $X$ projective, lisse, g\'eom\'etriquement connexe sur $\C((t))$.  On verra  \`a la remarque \ref{25suffitpas} que, pour $E$ un espace principal homog\`ene sous un $K$-tore $Q$
comme ci-dessus, on peut avoir   $E(K_{v}) \neq \emptyset$ pour tout compl\'et\'e de
$K$ en une valuation discr\`ete, pas d'obstruction de r\'eciprocit\'e 
comme dans la proposition 2.5, par rapport aux points ferm\'es de la $\C((t))$-courbe $X$,
et
n\'eanmoins $E(K)=\emptyset$.
 
Le corollaire  \ref{arbreimpliquepoint} ci-dessous explique pourquoi on ne peut s'attendre \`a une telle situation dans le cas
de la proposition \ref{pascontrex}, o\`u  $X=\P^1_{\C((t))}$.

\end{rema}

\begin{rema}
Supposons   $k=\F$ fini.
Si  $c \in F^{\times}$ n'est pas
un  carr\'e, 
il existe un \'el\'ement $d \in F^{\times}$ tel que 
les classes de $c$ et $d$ aient un cup-produit non nul dans $H^2(F,\Z/2)$.
Notons  $B= A \cup (d) \in H^3_{nr}(K(E)/K,\Z/2)$.
On trouve que pour toute famille $\{P_{\gamma} \in E(K_{\gamma})\}$, $\gamma$  parcourant les points ferm\'es
de $\P^1_{F}$, 
$$ \sum_{\gamma} \Cores_{F(\gamma)/F} \partial_{\gamma}(B(P_{\gamma})) = c \cup d  \neq 0  \in H^2(F,\Z/2)=\Z/2.$$
Ainsi  l'application de la proposition  \ref{weilrecip}  \`a $B \in H^3_{nr}(K(E)/K,\Z/2)$ 
permet aussi d'\'etablir $E_{c}(K)=\emptyset$.
On trouve ainsi  une illustration des r\'esultats g\'en\'eraux de Harari et Szamuely \cite{HaSz}.
 \end{rema}

\section{Calculs locaux}\label{evallocaledeA}

\subsection{Existence de points rationnels de $E$  sur un corps valu\'e discret}

Soient $R$ un anneau de valuation discr\`ete,  $\kappa$ son corps r\'esiduel
suppos\'e de caract\'eristique diff\'erente de 2, et $K$ son corps des fractions.
On note $v : K^{\times} \to \Z$ la valuation associ\'ee.

\begin{lem}\label{facile}
Soit $d \in K^{\times}$.  

(i) Si $d  \in  K^{\times 2}$, alors tout \'el\'ement de  $K^{\times}$
peut s'\'ecrire sous la forme $x^2-dy^2$ avec $x,y \in K$.

Supposons $R$ hens\'elien.

(ii)  Si   $v(d)$ est paire et $d$ n'est pas un carr\'e dans $R$,
alors pour $x,y \in K$ avec $x^2-dy^2\neq 0$, la valuation $v(x^2-dy^2)$ est paire.

(iii) Si $v(d)$ est impaire,   alors
 les classes
dans $K^{\times}/K^{\times 2}$ des
valeurs  prises par
 $x^2-dy^2$ pour $x,y \in K$ avec $x^2-dy^2\neq 0$ sont exactement $1$ et $-d$.
 
 (iv) Si $v(d)$ est paire et $d$ n'est pas un carr\'e dans $R$,
 si de plus 
le corps $\kappa$ satisfait $\cd_{2}\kappa \leq 1$,  alors les
valeurs  prises par
 $x^2-dy^2$  pour $x,y \in K$ avec $x^2-dy^2\neq 0$
sont exactement les produits
 d'une unit\'e et d'un carr\'e de~$K$.
\end{lem}
 \begin{proof}  
 Les \'enonc\'es (i) \`a (iii) sont clairs.
 Pour l'\'enonc\'e (iv), il suffit de remarquer que sous
 l'hypoth\`ese $\cd_{2}\kappa \leq 1$,  toute 
 \'equation  $x^2-\alpha y^2= \beta $ avec $\alpha, \beta \in \kappa^{\times}$ poss\`ede une solution
 avec $x,y \in \kappa$.
\end{proof}

\begin{prop}\label{pointslocaux}
Supposons $R$ hens\'elien.
Soient $a,b,c \in K^{\times}$.
  On consid\`ere  la $K$-vari\'et\'e $E$ d\'efinie par  l'\'equation
$$(X_{1}^2-aY_{1}^2)(X_{2}^2-bY_{2}^2)(X_{3}^2-abY_{3}^2)=c.$$

(a) Si les trois conditions suivantes sont  simultan\'ement satisfaites :

(i) les \'el\'ements $a$ et $b$ ont une valuation paire,

(ii)  aucun de $a$, $b$, $ab$,  n'est un carr\'e
dans $K$,

(iii)  la valuation de $c$ est impaire,

alors $E(K)=\emptyset$. 

(b) Si l'on a   $\cd_{2}(\kappa) \leq 1$,
la r\'eciproque est vraie.
\end{prop}
\begin{proof} Ceci se d\'eduit ais\'ement du lemme \ref{facile}.
\end{proof}

\subsection{Valeurs prises par l'alg\`ebre $A \in \Br E$
  sur un corps valu\'e discret}\label{valeursvaldisc}

Soit $R$ un anneau de valuation discr\`ete de rang 1 avec $2 \in R^{\times}$.
Soient $K$  le corps des fractions de $R$ et $\kappa$ son corps r\'esiduel.
On note $\pi$ une uniformisante de $R$.

 On a la suite exacte bien connue (cf. \cite[III, Prop. (2.1)]{GBBr})
$$0 \to {}_{2}\Br R \to {}_{2}\Br K \to \kappa^{\times}/\kappa^{\times 2} \to 1.$$
On note $\partial : {}_{2}\Br K \to \kappa^{\times}/\kappa^{\times 2}$ l'application r\'esidu.
Il est bien connu \cite[\S 1]{kato}  qu'elle  envoie la classe d'une alg\`ebre de quaternions $(a,b)$ sur la classe
$$(-1)^{v(a)v(b)} \overline{a^{v(b)}/b^{v(a)} } \in  \kappa^{\times}/\kappa^{\times 2} .$$

Si $R$ est hens\'elien, alors $ {}_{2}\Br R \oi  {}_{2}\Br \kappa$.
Si $R$ est hens\'elien et $\cd_{2} (\kappa) \leq 1$, alors
$\partial : {}_{2}\Br K \oi \kappa^{\times}/\kappa^{\times 2}.$

 \begin{prop}\label{valeursdeA}
  Soient $a,b,c \in K^{\times}$.
Soit $E$  la $K$-vari\'et\'e $E$ d\'efinie par
$$(X_{1}^2-aY_{1}^2)(X_{2}^2-bY_{2}^2)(X_{3}^2-abY_{3}^2)=c.$$
 Soit 
$A = (X_{1}^2-aY_{1}^2,b) \in \Br E.$
Soit $ev_{A} : E(K) \to \Br K$ l'application donn\'ee par $P \mapsto A(P)$,
et
soit $\Delta$ l'application compos\'ee
$$ E(K) \to {}_{2}\Br K \to \kappa^{\times}/\kappa^{\times 2}.$$

$$ P \mapsto   \partial(A(P)).$$

(i) Si  $b$ ou $ab$ est un carr\'e dans $K$, l'image de $ev_{A}$ est $0$.

(ii) Si $a$ est un carr\'e dans $K$, l'image de $ev_{A}$ est $(c,b)$;
si de plus $c$ et $b$ sont de valuation paire,
 l'image de $\Delta$ est $1$.

(iii)   Supposons $R$ hens\'elien.
Si $v(a)$ et $v(b)$ sont paires et $a$ n'est pas un carr\'e dans $K$,
et si $E(K)  \neq \emptyset$, 
l'image de $\Delta$ est $1$.

 (iv)  Supposons $R$ hens\'elien.
 Supposons qu'aucun de $a,b,ab$ n'est un carr\'e dans $K$,
 et que $v(a)$ et $v(b)$ ne sont pas toutes deux paires.
 
 Si $v(a)$ est paire et $v(b)$ impaire, l'image de $ev_{A}$
 est contenue dans $\{ (c,b), (ac,b) \}$.

 Si $v(a)$ est impaire et $v(b)$ paire, l'image de 
 $ev_{A}$ est contenue dans $\{ 0, (-a,b) \}$.

 Si $v(a)$ et $v(b)$ sont impaires, 
 l'image de 
 $ev_{A}$ est contenue dans $\{ 0, (-a,b) \}$.
 
 L'image de $\Delta$
 est contenue dans la r\'eunion de deux classes (\'eventuellement confondues) dans
  $\kappa^{\times}/\kappa^{\times 2}$
qui diff\`erent par
multiplication par $(-1)^{v(a).v(b)}\partial((a,b))$.

(v)  Sous l'hypoth\`ese de (iv), si de plus 
   $\cd_{2}(\kappa) \leq 1$,  alors les classes indiqu\'ees ci-dessus
 sont toutes dans l'image de $ev_{A} $, resp. dans celle de  $\Delta$.
\end{prop}

 \begin{proof}

Supposons d'abord que l'un de $a$, $b$, $ab$ est  un carr\'e dans $K$,
auquel cas on a automatiquement $E(K) \neq \emptyset$.
Si $a \notin K^{\times 2}$, alors soit $b$ soit $ab$ est dans $K^{\times 2}$
et alors $A = 0 \in \Br E$ d'apr\`es (\ref{Avia1}) et (\ref{Avia3}), donc
  pour tout $P  \in E(K )$, on a $A(P)=0 \in \Br K$.
Si  $a \in K^{\times 2}$, alors $A = (c,b)\in \Br E$
d'apr\`es (\ref{Avia3}),
donc
pour tout $P  \in E( K )$,
 on a $A(P)=(c,b) \in \Br K$. Si $c$ et $b$ sont de valuation paire,
 ceci implique $\partial (A(P))=0$. Ceci \'etablit (i) et (ii).

\medskip

Supposons que aucun  de $a$, $b$, $ab$ n'est  un carr\'e dans $K$.

Si  $R$ est hens\'elien et si les valuations de $a$ et de $b$ sont paires, alors, 
pour tout $P \in E( K)$, de $A = (X_{1}^2-aY_{1}^2,b)$, en utilisant le lemme \ref{facile} (ii),
on d\'eduit
 $\partial (A(P))=0$, c'est-\`a-dire l'\'enonc\'e (iii).

 Nous ne d\'etaillons pas ici les calculs faits pour \'etablir (iv) et (v), qui utilisent le lemme \ref{facile}.
 On doit dans chaque cas discuter selon la parit\'e de la valuation de~$c$.
\end{proof}

\subsection{Obstruction de r\'eciprocit\'e attach\'ee \`a l'alg\`ebre $A \in \Br E $ en un point ferm\'e d'un anneau local r\'egulier
de dimension 2}\label{calculobsrecip}

Soit $R$ un anneau local r\'egulier excellent de dimension 2, 
avec $2\in R^{\times}$.
Pour $\gamma$ un point de codimension 1 de $\Spec R$ on note aussi
$\gamma$ sa fermeture dans $\Spec R$ et on note encore $\gamma \in R^{(1)}$
 l'id\'eal premier de hauteur $1$ de $R$ associ\'e.

 Notons $M$ le point ferm\'e de $\Spec R$.
C'est aussi le point ferm\'e de $\Spec R/\gamma$ pour tout point $\gamma$ de
codimension 1 de $\Spec R$.
Soit $K$ le corps des fractions de $R$. 

On a un complexe
$$ {}_{2}\Br K   \to \oplus_{\gamma \in R^{(1)} }    \kappa(\gamma)^{\times}/ \kappa(\gamma)^{\times 2} \to \Z/2,$$
avec
$$\partial_{\gamma} : {}_{2}\Br K \to   \kappa(\gamma)^{\times}/ \kappa(\gamma)^{\times 2}$$
et
$$\partial_{\gamma,M} :  \kappa(\gamma)^{\times}/ \kappa(\gamma)^{\times 2} \to \Z/2.$$

Soient $a,b,c \in K^{\times}$.

Soit
$$
 E : (X_1^2 - aY_1^2)(X_2^2 - bY_2^2)(X_3^2 - abY_3^2) = c
 $$

et
$$A = (X_{1}^2-aY_{1}^2,b)= (c(X_{3}^2-abY_{3}^2),b)  = (c(X_{2}^2-bY_{2}^2),ab)  \in \Br E.$$

Supposons $E(K_{\gamma})  \neq \emptyset $ en tout point de codimension 1 de $\Spec R$.

On veut calculer la valeur des sommes
 $$
\sum_{\gamma \in R^{(1)} } \partial_{\gamma,M}(\partial_\gamma(A(P_\gamma)))$$
pour une famille $\{P_{\gamma} \in E(K_{\gamma})\}_{\gamma \in R^{(1)}}$.

Pour simplifier, pour $\alpha \in \Br K_{\gamma}$, on note  ici
$$\partial_{M}(\partial_\gamma(\alpha))  = \partial_{\gamma,M}(\partial_\gamma(\alpha)).$$

\begin{lem}\label{unestinversible}
 Si l'un de $a$, $b$ ou $ab$ est le produit d'une unit\'e de $R$
par un carr\'e de $K$, alors pour toute famille $\{P_{\gamma}\} \in \prod_{\gamma} E(K_{\gamma})$,
on a 
$$\sum_{\gamma \in R^{(1)} } \partial_M(\partial_\gamma(A(P_\gamma)))=0 \in \Z/2.$$
\end{lem}

\begin{proof}
Si  $b$  est le produit d'une unit\'e de $R$
par un carr\'e de $K$, on peut supposer que $b$ est une unit\'e dans $R$.
On a 
$$
\sum_{\gamma \in R^{(1)}} \partial_M(\partial_\gamma(A(P_\gamma))) = \sum_{\gamma \in {R^{(1)}} }\partial_M(\partial_\gamma(X_1^2 -
aY_1^2, b)(P_{\gamma}))=  \sum_{\gamma \in {R^{(1)}}} \partial_M((\overline{b})^{\nu_\gamma(X_1^2 - aY_1^2)(P_{\gamma})} ).
$$

Puisque $b$ est une unit\'e, la classe $\overline b$ est une unit\'e dans $R/\gamma$
 donc  $\partial_{M}(\overline{b})=0$, et
$$\sum_{\gamma \in {R^{(1)}}} \partial_M(\partial_\gamma(A(P_\gamma))) =0.$$

\medskip

Si $ab$ est le produit d'une unit\'e de $R$
par un carr\'e de $K$, on peut supposer que $ab$ est une unit\'e dans $R$.
En utilisant $A=(X_{1}^2-aY_{1}^2,ab)$, le m\^eme argument que ci-dessus donne
$$\sum_{\gamma \in {R^{(1)}}} \partial_M(\partial_\gamma(A(P_\gamma))) =0.$$
\medskip

Si $a$ est le produit d'une unit\'e de $R$
par un carr\'e de $K$, on peut supposer que $a$ est une unit\'e dans $R$,
le m\^eme argument que ci-dessus donne
$$\sum_{\gamma \in R^{(1)}} \partial_M(\partial_\gamma(A_{1}(P_\gamma))) =0$$
pour $A_{1}= (X_{2}^2-bY_{2}^2,a)$. Comme on a
$$A=A_{1} + (c,ab),$$
on obtient
 $$\sum_{\gamma \in R^{(1)}}    
 \partial_M(\partial_\gamma(A(P_\gamma)) =
\sum_{\gamma \in R^{(1)}} 
\partial_{M}(\partial_{\gamma}(c,ab)),$$
et le second terme est  nul par la loi de r\'eciprocit\'e
appliqu\'ee \`a $(c,ab) \in {}_2 \Br K$.
\end{proof}

\begin{lem}\label{carrelocallocal} 
Soit $R$ un anneau local r\'egulier de dimension 2 et $(\pi,\delta )$ un syst\`eme r\'egulier
de param\`etres. Si $z \in R$ non nul
a son diviseur support\'e sur la r\'eunion des diviseurs de  $\pi$ et de $\delta $, 
et si  c'est un carr\'e dans le corps des fractions du hens\'elis\'e  du localis\'e de $R$
le long de $\pi$, alors $z=u\pi^{2r}\delta^{2s} \in R$  avec $u\in R^{\times}$ d'image un carr\'e dans
$R/\pi$, donc d'image un carr\'e dans $R/m$. 
\end{lem}

\begin{proof}
C'est clair.
\end{proof}

\begin{prop}\label{partoutcarrelocal}
Si   la r\'eunion des diviseurs r\'eduits  de $a$ et de $b$
 est un diviseur \`a croisements normaux, et si en tout $\gamma \in R^{(1)} $ l'un de $a$, $b$ ou $ab$
est un carr\'e dans $K_{\gamma}$,
alors pour toute famille $\{P_{\gamma}\} \in \prod_{\gamma  \in R^{(1)}} E(K_{\gamma})$,
on a 
$$\sum_{\gamma \in R^{(1)} } \partial_M(\partial_\gamma(A(P_\gamma)))=0.$$
\end{prop}

\begin{proof}
Cela r\'esulte des lemmes \ref{unestinversible}  et \ref{carrelocallocal}.
\end{proof}

\begin{prop}\label{toutZmod2}
Supposons que  la r\'eunion des diviseurs r\'eduits  de $a$ et de $b$
et de $c$
est contenue dans un diviseur \`a croisements normaux d\'efini
par un syst\`eme r\'egulier de param\`etres $(\pi,\delta )$.

On a alors les propri\'et\'es suivantes.

(i)  Pour toute famille $\{P_{\gamma  }\} \in \prod_{\gamma  \in R^{(1)}} E(K_{\gamma})$,
$$
\sum_{\gamma \in {R^{(1)}}}  \partial_M(\partial_\gamma(A(P_\gamma)))
= \partial_M(\partial_\pi(A(P_\pi))) + \partial_M(\partial_\delta(A(P_\delta))) \in \Z/2.
$$

Supposons qu'aucun de $a$, $b$  ou $ab$ 
n'est le produit d'une unit\'e de $R$ par un carr\'e dans $K$.

(ii) Si  le corps des fractions $\kappa(\pi)$ de $R/\pi$
satisfait 
$\cd_{2}  \kappa(\pi ) \leq 1$, alors
pour $P_{\pi}$ variant dans $E(K_{\pi})$,
 $\partial_{M}(\partial_\pi(A(P_\pi)))$ prend les deux valeurs dans $\Z/2$.
 
 (iii)  Si  le corps des fractions  
 $\kappa(\delta)$ de $R/\delta$ satisfait 
$\cd_{2}  \kappa(\delta ) \leq 1$, alors
pour $P_{\delta }$ variant dans $E(K_{\delta})$,
$\partial_{M}( \partial_\delta (A(P_{\delta} )))$ prend les deux valeurs dans $\Z/2$.

\end{prop}

\begin{proof}
 
 Soit $\gamma$ un point de codimension 1 de $R$ autre que $\pi$ et $\delta $.
 Soit $K_{\gamma}$ le corps des fractions du hens\'elis\'e de $R$ en $\gamma$.
 En un tel point, $v_{\gamma}(a)=0$, $v_{\gamma}(b)=0$ et $v_{\gamma}(c)=0$. 
 La proposition \ref{valeursdeA} montre alors  que pour tout $P_{\gamma} \in E(K_{\gamma})$,  
on  a  $\partial_{\gamma}(A(P_{\gamma}))=0$. Ceci donne l'\'enonc\'e (i).

 Si  aucun de $a$, $b$  ou $ab$ 
n'est le produit d'une unit\'e de $R$ par un carr\'e dans $K$,
 quitte \`a se d\'ebarrasser de carr\'es et \`a modifier le syst\`eme r\'egulier
de param\`etres en les multipliant par des unit\'es, 
on peut supposer que  l'on est dans l'un des cas suivants  :
$(a,b)=(\pi,\delta )$, ou 
$(a,b)=(\pi,\pi\delta )$, ou
$(a,b) =(\pi\delta ,\delta)$, 
et   $c=u\pi^r\delta ^s$ avec $u\in R^{\times}$ et $r,s \in \Z$.

Supposons que le corps des fractions $\kappa(\pi)$ de $R/\pi$ satisfait 
$\cd_{2}  \kappa(\pi ) \leq 1$.
La proposition \ref{valeursdeA}  appliqu\'ee
  au hens\'elis\'e de l'anneau de valuation discr\`ete $R_{\pi}$  
montre que
$ \partial_\pi(A(P_\pi))$, pour $P_{\pi}$ variant dans $E(K_{\pi})$,
prend deux valeurs dans $\kappa(\pi)^{\times}/\kappa(\pi)^{\times 2}$
qui diff\`erent par multiplication par $\pm \overline {\delta } \neq 1$.
Ainsi $\partial_{M}( \partial_\pi(A(P_\pi))) \inÊ\Z/2$,
pour $P_{\pi}$ variant dans $E(K_{\pi})$,
prend les deux valeurs dans $\Z/2$. Ceci donne l'\'enonc\'e (ii),
et la d\'emonstration de (iii) est analogue.

 \end{proof}

\subsection{Comportement dans un \'eclatement}\label{eclatement}

Cette section
sera utilis\'ee au paragraphe \ref{paradescente}.

\begin{prop}\label{eclat}
Soit  $R$ un anneau  local r\'egulier 
 int\`egre de dimension~2, de corps des fonctions $K$,
de corps r\'esiduel $k$ s\'eparablement clos, de caract\'eristique diff\'erente de 2.
Soit $p = \mathcal Y \to {\mathcal  X} $ l'\'eclatement de  ${\mathcal  X} = \Spec R$ en le point ferm\'e $M$ de  ${\mathcal  X} $
et soit $\lambda \in {\mathcal Y}^{(1)}$ la courbe exceptionnelle introduite dans l'\'eclatement.
 
Soient $a,b,c \in K^{\times}$, puis
$$
 E : (X_1^2 - aY_1^2)(X_2^2 - bY_2^2)(X_3^2 - abY_3^2) = c
 $$
et
$A = (X_{1}^2-aY_{1}^2,b)  \in \Br E.$

Supposons   que  la r\'eunion des diviseurs r\'eduits  de $a$ et de $b$ et de $c$
est un diviseur \`a croisements normaux sur $\Spec R$.
 
Soit $\{P_{\gamma} \in E(K_{\gamma}) \}_{ \gamma \in {\mathcal X}^{(1)}}$ une famille de points telle que 
 $$
\sum_{\gamma \in {\mathcal X}^{(1)}} \partial_M(\partial_\gamma(A(P_\gamma)))=0 \in \Z/2.$$
 
 Il existe alors un point $P_{\lambda} \in E(K_{\lambda})$ tel que  pour tout point ferm\'e
 $N \in   {\mathcal Y}$ on ait 
 $$
\sum_{\zeta \in  {\mathcal Y}^{(1)}, \ N \in \zeta} \partial_N(\partial_\zeta(A(P_\zeta)))=0 \in \Z/2,$$
o\`u, pour $\zeta \in {\mathcal Y}^{(1)}, \zeta\neq \lambda,$ on note $P_{\zeta}=P_{p(\zeta)}$.
\end{prop}

\begin{proof}
 Il suffit d'\'etablir l'\'enonc\'e pour les
points $N \in \lambda$. Toutes les \'egalit\'es qui suivent sont dans $\Z/2$.

Par hypoth\`ese, la  r\'eunion des diviseurs r\'eduits  de $a$ et de $b$ et de $c$
est un diviseur \`a croisements normaux sur $\Spec R$.
 On choisit 
 $(\pi,\delta ) \in R $ des g\'en\'erateurs de l'id\'eal maximal de $R$
tels que chacun de $a$ et $b$
 soit produit d'une unit\'e en $M$
par des puissances de $\pi$ et de $\delta $.

Soit $S_{\lambda}$ le hens\'elis\'e de l'anneau local de $\mathcal Y$ en $\lambda$ et $K_{\lambda}$
l'anneau des fractions de $S_{\lambda}$. Le corps r\'esiduel de 
$S_{\lambda}$ est $k({\mathbf P}^1)$ qui satisfait $\cd_{2} k({\mathbf P}^1) \leq 1$
 par l'hypoth\`ese faite sur $k$.

L'\'eclat\'e au voisinage de $M$ est
obtenu par recollement de $\Spec R[t]/(\pi-t\delta  )$ et $\Spec R[s]/(\delta -s\pi)$.
Dans la premi\`ere carte, $\lambda$ est donn\'e par $\delta=0$,
dans la seconde par $\pi=0$.

Si l'un de $a$, $b$ ou $ab$ est une unit\'e fois un carr\'e en $M$,
alors, comme $k$ est s\'eparablement clos, c'est un carr\'e dans $S_{\lambda}$, 
donc $E(K_{\lambda}) \neq \emptyset$.
Par ailleurs, sous cette m\^eme hypoth\`ese, en tout point $N \in \lambda$,  l'un de 
$a$, $b$ ou $ab$ est   une unit\'e fois un carr\'e en $N$,
le lemme \ref{unestinversible} montre
que pour tout point $N \in \lambda$ et 
pour toute famille    de points $\{P_{\zeta} \in E(K_{\zeta})\}_{ \zeta \in  {\mathcal Y}^{(1)}}$ on a
$$
\sum_{\zeta \in  {\mathcal Y}^{(1)}, N \in \zeta} \partial_N(\partial_\zeta(A(P_\zeta)))=0.$$

Si aucun de $a$, $b$ ou $ab$ n'est une unit\'e fois un carr\'e en $M$,
 on peut supposer que l'on a : $(a,b)=(\pi,\delta )$
ou  $(a,b)=(\pi,\pi\delta )$ ou $(a,b)=(\pi\delta ,\delta)$.
Dans la premi\`ere carte,  le couple
$(a,b)$ est, \`a multiplication par des carr\'es pr\`es, \'egal 
\`a l'un de $(t\delta ,\delta ), (t\delta ,t), (t,\delta )$. L'une des coordonn\'ees
$a, b$ a une $\delta$-valuation impaire.
La proposition \ref{pointslocaux} donne  $E(K_{\lambda}) \neq \emptyset$.

Pour $\gamma$ passant par $M$ et diff\'erent de $\pi$ ou $\delta $, 
$ a$, $b$ et $c$  sont des unit\'es dans $R_{\gamma}$, ce qui d'apr\`es
la proposition \ref{valeursdeA} implique
  $\partial _{\gamma}(A(P_{\gamma}))=0$.

L'hypoth\`ese
 $$
\sum_{\gamma \in {\mathcal X}^{(1)}, M \in \gamma } \partial_M(\partial_\gamma(A(P_\gamma)))=0$$
implique alors
\begin{equation}\label{(*)}
\partial_M(\partial_\pi(A(P_\pi)))= \partial_M(\partial_\delta (A(P_\delta ))). 
\end{equation}

Soient $N_{1}$, resp. $N_{2}$, le point  de $\lambda \subset \mathcal{Y}$ 
o\`u le transform\'e propre de $\pi=0$ (donn\'e par $t=0$), resp.  de $\delta =0$ (donn\'e par $s=0$),
rencontre $\lambda$.

Il nous faut montrer qu'il existe $P_{\lambda} \in E(K_{\lambda})$ tel que l'on ait
$$\partial_{N_{1}}(\partial_{\lambda}(A(P_{\lambda})))= \partial_{N_{1}}(\partial_{\pi}(A(P_{\pi}))) $$
et
$$\partial_{N_{2}}(\partial_{\lambda}(A(P_{\lambda})))=  \partial_{N_{2}}(\partial_{\delta}(A(P_{\delta}))) $$
et enfin
 $$\partial_{N}(\partial_{\lambda}(A(P_{\lambda})))=0  $$
pour $N \in \lambda$, $N \neq N_{1}, N_{2}$.

On  a
\begin{equation}\label{(*a)}
\partial_{N_{1}}(\partial_{\pi}(A(P_{\pi}))) = \partial_{M}(\partial_{\pi}(A(P_{\pi})))
\end{equation}
et
\begin{equation}\label{(*b)}
\partial_{N_{2}}(\partial_{\delta}(A(P_{\delta})))=\partial_{M}(\partial_{\delta }(A(P_{\delta}))).
\end{equation}

La proposition \ref{valeursdeA} et
les hypoth\`eses sur $a$, $b$, $c$ impliquent 
   que   pour tout point $P_{\lambda} \in E(K_{\lambda}) $,    $\partial_{\lambda}(A(P_{\lambda}))$ 
est   de la forme $\partial_{\lambda}((u,v))$
avec chacun de $u,v$ produit d'une unit\'e en $M$ et de puissances de $\pi$ 
et de $\delta $.  
Une consid\'eration d'une des cartes locales de $\mathcal Y$ 
 montre imm\'ediatement
que pour une telle alg\`ebre de quaternions $(u,v)$ on~a 
$$\partial_{N}(\partial_{\lambda}(u,v))=0$$
pour $N \in \lambda$, $N  \neq N_{1}, N_{2}$. Ainsi, quel que soit le choix de $P_{\lambda}$, on~a
$$\partial_{N}( \partial_{\lambda} (A(P_{\lambda}) ))=0$$
pour $N \in \lambda$, $N \neq N_{1}, N_{2}$.

La loi de r\'eciprocit\'e sur la courbe $\lambda$  et l'\'el\'ement $\partial_{\lambda}(A(P_{\lambda}))  \in 
\kappa(\lambda)^{\times}/\kappa(\lambda)^{\times 2}$
donne $$\sum_{N \in \lambda} \partial_{N} (\partial_{\lambda} (A(P_{\lambda}) ))=0.$$

On a donc pour tout $P_{\lambda}$
\begin{equation}\label{(**)}
\partial_{N_{1}}(\partial_{\lambda} (A(P_{\lambda}) )) + \partial_{N_{2}} (\partial_{\lambda} (A(P_{\lambda}) ))=0.
\end{equation}

 Dans la premi\`ere carte de $\mathcal Y$, o\`u $\lambda$ est donn\'e par $\delta =0$,
 sur lequel $t =0$ (transform\'e propre du diviseur $\pi=0$ sur $\mathcal X$)
 d\'ecoupe le diviseur $N_{1}$
on trouve,   \`a multiplication par des carr\'es pr\`es,  l'\'egalit\'e de couples 
$(a,b)=(t\delta,\delta )$ 
ou $(a,b)=(t\delta, t )$  ou $(a,b)=(t,\delta  )$
et l'id\'eal maximal dans l'anneau local au point $N_{1}$  est engendr\'e
par $(t,\delta)$.

 Comme le corps r\'esiduel $\kappa(\lambda)$ 
  de $S_{\lambda}$ satisfait $\cd_{2}\kappa(\lambda) \leq 1$,
  la proposition \ref{valeursdeA}  
   appliqu\'ee au voisinage
du  point $N_{1}$ de $\mathcal Y$ 
  montre 
  que l'on peut trouver $P_{\lambda} \in E(K_{\lambda})$ tel que
  \begin{equation}\label{(***)}
\partial_{N_{1}}(\partial_{\lambda} (A(P_{\lambda}) ))  =\partial_{N_{1}}(\partial_{\pi}(A(P_{\pi}))).  
\end{equation}

Les \'egalit\'es (\ref{(*)}), (\ref{(*a)}),  (\ref{(*b)}),    (\ref{(**)})  et (\ref{(***)})
impliquent alors  
$$\partial_{N_{2}} (\partial_{\lambda} (A(P_{\lambda}) )) =  \partial_{N_{2}}(\partial_{\delta }(A(P_{\delta})))   .$$
\end{proof}

\section{Le principe local-global ne vaut pas en g\'en\'eral}\label{coronidisloco}

 \subsection{Un \'enonc\'e semi-local}
 
\begin{prop}\label{semilocal}
Soit $R$ un anneau semi-local r\'egulier  int\`egre de dimension $2$,
avec $2 \in R^{\times}$. Supposons que $R$ a exactement trois
id\'eaux maximaux $m_1, m_2,
 m_3$, et que $m_1 = (\pi_2, \pi_3)$,
 $m_2 = (\pi_1, \pi_3)$, $m_3 = (\pi_1, \pi_2)$,
chacun des $(\pi_{i})$  \'etant un id\'eal premier.
Soit $K$ le corps des fractions de $R$.
Soient  $a = \pi_2\pi_3$, $b = \pi_3\pi_1$ 
et $ c = \pi_1\pi_2\pi_3$.   Alors la $K$-vari\'et\'e 
d'\'equation 
$$
 (X_1^2 - aY_1^2)(X_2^2 - bY_2^2)(X_3^2 - abY_3^2) = c
 $$
 n'a pas de point rationnel sur $K$.
\end{prop}

\begin{proof}
 Soit $K_{i}$ le corps des fractions du hens\'elis\'e du localis\'e de  $R$ 
 en l'id\'eal premier $\pi_{i}$. Soit $\kappa_{i}$ le corps des fractions de $R/\pi_{i}$,
 corps r\'esiduel de ce hens\'elis\'e.
 D'apr\`es la proposition \ref{valeursdeA},
  l'image de l'application compos\'ee
$$E(K_{i}) \to {}_{2}\Br K_{i} \to \kappa_{i}^{\times}/ \kappa_{i}^{\times 2},$$
o\`u la premi\`ere fl\`eche est l'\'evaluation de $A=(X_{1}^2-aY_{1}^2, b)$
et la seconde l'application r\'esidu,
est contenue :

pour $i=1$, dans $\{-\pi_{2}, -\pi_{3}\}$;

pour $i=2$,  dans $\{1,\pi_{1}\pi_{3} \}$;

pour $i=3$, dans $\{1, \pi_{1}\pi_{2} \}$.

On consid\`ere l'application compos\'ee

$$\prod_{i} E(K_{i}) \to  \bigoplus _{i=1}^{i=3} {}_{2}\Br K_{i} \to
\bigoplus_{i=1}^{i=3}  \kappa_{i}^{\times}/ \kappa_{i}^{\times 2}
\to \bigoplus_{j=1}^{j=3}  \Z/2$$
d\'efinie \`a partir des applications d'\'evaluations de $A$
et du complexe de Bloch-Ogus sur l'anneau semi-local $R$.
Les indices $j$ sont ceux des points ferm\'es $m_{j}$.
L'application 
$ \kappa_{i}^{\times}/ \kappa_{i}^{\times 2} \to \oplus_{j=1}^{j=3}  \Z/2$
envoie une classe $\gamma \in  \kappa_{i}^{\times}/ \kappa_{i}^{\times 2}$
sur la valuation modulo $2$ de $\gamma$
en $m_{j}$ pour $j\neq i$ et sur $0$ 
en $m_{i}$.

Pour tout id\'eal premier $x$ de $R$ autre que
ceux d\'efinis par
$\pi_{1}, \pi_{2}, \pi_{3}$, les \'el\'ements
$a, b,c$ sont des unit\'es dans $R_{x}$.
La proposition \ref{valeursdeA} montre qu'en un tel $x$,
pour  tout point   $P_{x} \in E(K_{x})$, on a
$\partial_{x}(A(P_{x}))=0$.

Ceci implique que l'application compos\'ee

$$E(K) \to \prod_{i} E(K_{i}) \to  \bigoplus _{i=1}^{i=3} {}_{2}\Br K_{i} \to
\bigoplus_{i=1}^{i=3}  \kappa_{i}^{\times}/ \kappa_{i}^{\times 2}
\to \bigoplus_{j=1}^{j=3}  \Z/2$$
est nulle.

Pour $i=1$, l'image dans $ \bigoplus_{j=1}^{j=3}  \Z/2$
de $(-\pi_{2})$, resp. $(-\pi_{3})$ dans $ \kappa_{1}^{\times}/ \kappa_{1}^{\times 2}$,
consiste en $(0,0,1)$, resp. $(0,1,0)$.  

Pour $i=2$,  l'image dans $ \bigoplus_{j=1}^{j=3}  \Z/2$
de $1$, resp. $\pi_{1}\pi_{3}$, consiste en $(0,0,0)$, resp. $(1,0,1)$.

Pour $i=3$, l'image dans $ \bigoplus_{j=1}^{j=3}  \Z/2$
de $1$, resp. $\pi_{1}\pi_{2}$ consiste en $(0,0,0)$, resp. $(1,1,0)$.

Aucune addition sur $i=1,2,3$ de ces triplets (un par indice $i$)  ne donne $(0,0,0)$.
Ainsi $E(K)=\emptyset$.
\end{proof}

\subsection{Contre-exemple au principe local-global pour les valuations}

Soit $\mathcal X$ un sch\'ema  int\`egre  r\'egulier excellent 
de dimension deux
avec $2 $ inversible sur $\mathcal X$, quasi-projectif sur un anneau,
satisfaisant la  condition suivante :

{\it  Il existe un triplet de diviseurs int\`egres lisses $L_{i} \subset {\mathcal X}$
formant triangle :  pour $i\neq j \in \{1,2,3\}$, $L_{i} \cap L_{j}$ est une intersection transverse
en un point ferm\'e $m_{k}$, o\`u $\{i,j,k\}=\{1,2,3\}$ et les points $m_{k}, k \in 
 \{1,2,3\}$ sont distincts.}
 
 \medskip
 
 Soit $K$ le corps des fonctions de  $\mathcal X$.
 
Comme tout anneau semi-local  a un groupe de Picard trivial,
il existe des \'el\'ements $\pi_{i} \in K^{\times}, i=1,2,3,$
de diviseurs 
 $$div_{\mathcal X}(\pi_{i}) =L_{i} + D_{i}$$
satisfaisant  les conditions suivantes.

(a) Le support du diviseur $D_{1}$ ne contient 
 aucun des $m_{i}$ (et donc aucun des $L_{i}$).
 
(b)  Le support du diviseur $D_{2}$ ne contient 
 aucun des $m_{i}$ (et donc aucun des $L_{i}$),
ni   aucune composante du support de $D_{1}$,
  ni aucun des points de rencontre des $L_{i}$
 avec un point du support de $D_{1}$.
 
 (c) Le support du diviseur $D_{3}$ ne contient 
 aucun des $m_{i}$ (et donc aucun des $L_{i}$),
ni   aucune composante du support de $D_{1}$ ou de $D_{2}$,
 ni aucun des points de rencontre des $L_{i}$
 avec un point du support de $D_{1}$, ni aucun des points 
de rencontre des $L_{i}$ avec un point du support de $D_{2}$,
ni aucun point de rencontre des supports des diviseurs $D_{1}$ et $D_{2}$.

Ces conditions impliquent  que 
pour tout point ferm\'e $M$  de $\mathcal X$
l'un des $\pi_{i}$ est inversible en $M$.

\begin{prop}\label{contreexempleglobalgeneral}
Avec les notations ci-dessus, soient  $a = \pi_2\pi_3$, $b = \pi_3\pi_1$ 
et $ c = \pi_1\pi_2\pi_3$. 
Soit $E$ la $K$-vari\'et\'e
 d'\'equation 
$$
 (X_1^2 - aY_1^2)(X_2^2 - bY_2^2)(X_3^2 - abY_3^2) = c
 $$

(i) La $K$-vari\'et\'e    $E$ n'a pas de point $K$-rationnel.

(ii) Supposons que  les corps r\'esiduels des points de codimension 2 sont s\'eparablement clos,
et que les corps r\'esiduels des points de codimension 1 de $\mathcal X$ sont de 2-dimension cohomologique
inf\'erieure ou \'egale  \`a $1$. Alors pour tout  anneau  local hens\'elien int\`egre  $R$, de corps des fractions $F$,
 et tout morphisme dominant
 $\Spec R \to \mathcal X$,  on a $E(F) \neq \emptyset$.
\end{prop}

  \begin{proof}
Soit $B$ l'anneau semi-local de $\mathcal X$ dont les points ferm\'es
 sont $m_{1},m_{2},m_{3}$. Sur cet anneau, la situation est exactement celle
 d\'ecrite dans  la proposition \ref{semilocal}.   Ceci \'etablit (i).

 Soit $\gamma$ un point de codimension 1 de $\mathcal X$.
 On a 
$$div_{\mathcal X}(a)= L_{2} +L_{3} + D_{2} + D_{3},$$
 $$div_{\mathcal X}(b)= L_{3} +L_{1} + D_{3} + D_{1},$$
 $$div_{\mathcal X}(c)= L_{1} +L_{2} + L_{3} + D_{1} + D_{2} + D_{3}.$$
 
 Si l'on a $v_{\gamma}(a)$ pair, alors $\gamma$ est une des composantes du support de 
   $D_{2}$
 ou $D_{3}$, o\`u elle appara\^{\i}t avec   multiplicit\'e paire. Mais alors $v_{\gamma}(c)$ est pair.
  De la proposition \ref{pointslocaux} et de l'hypoth\`ese $\cd_{2}(\kappa(\gamma) ) \leq 1$
   il r\'esulte alors :
 
 \medskip
 
 {\it Pour tout point $\gamma$ de codimension 1 de $\mathcal X$, on a $E(K_{\gamma}) \neq \emptyset$.}
 
 \medskip
 
 Soient $R$ un anneau  local hens\'elien int\`egre, de corps des fractions $F$,
 et
 $\Spec R \to \mathcal X$  un morphisme dominant.

 Soit $x$  l'image du point ferm\'e de $R$. Ceci induit une inclusion locale
 $O_{M,x} \hookrightarrow R$ et donc une inclusion locale
  $O_{M,x}^h  \hookrightarrow R$  au niveau des hens\'elis\'es,
  puis une inclusion des corps de fractions $K_{x} \hookrightarrow F$.

  Si $x=\gamma$  est un point de codimension 1 de $\mathcal X$,
  on vient d'\'etablir  $E(K_{\gamma}) \neq \emptyset$,  on a donc $E(F) \neq \emptyset$.  
  
 Supposons que $x=M$ est un point ferm\'e de  $\mathcal X$.
 Les conditions impos\'ees aux $D_{i}$ garantissent, comme on le v\'erifie
 ais\'ement, que l'un au moins  des $\pi_{i}$, soit $\pi_{i_{0}} $, est inversible en $x$.
 Comme le corps r\'esiduel de $M$ est s\'eparablement clos de caract\'eristique diff\'erente de 2, la classe
 de $-\pi_{i_{0}} $ dans ce corps r\'esiduel est un carr\'e, elle l'est donc aussi dans le corps r\'esiduel de 
 l'anneau local hens\'elien $R$, et donc
    $-\pi_{i_{0}} $ est un carr\'e dans  $R \subset F$. 
Ainsi  l'image  de $\pi_{1}\pi_{2}\pi_{3}$  dans $F$
est le produit d'un carr\'e  de $F$ et du produit $-\prod_{i\neq i_{0}} \pi_{i}$.

L'\'equation
 $$
 (X_1^2 - \pi_{2}\pi_{3}Y_1^2)(X_2^2 - \pi_{3}\pi_{1}Y_2^2)(X_3^2 - \pi_{1}\pi_{2}\pi_{3}^2Y_3^2) = \pi_{1}\pi_{2}\pi_{3}
 $$
 admet sur $F$ une solution avec toutes les coordonn\'ees sauf $Y_{i_{0}}$
 nulles, ce qui ach\`eve la d\'emonstration.
   \end{proof}

\begin{cor}\label{contrexcourbe}
 Soit $R$ un anneau de valuation discr\`ete hens\'elien, excellent, de corps
 r\'esiduel s\'eparablement clos de caract\'eristique diff\'erente de 2,  et
 ${\mathcal X}/R$ un $R$-sch\'ema r\'egulier int\`egre, projectif sur $R$,
\`a fibre g\'en\'erique une courbe lisse  g\'eom\'etriquement int\`egre.
Soit $K$ le corps des fonctions de ${\mathcal X}$.
Supposons que la fibre sp\'eciale r\'eduite de ${\mathcal X}/R$
contient un triangle $L_{1},L_{2},L_{3}$.
Choisissons les $\pi_{i} \in  K$  et $a,b,c \in K$ comme dans
la proposition \ref{contreexempleglobalgeneral}.
Soit $E$ la $K$-vari\'et\'e d'\'equation
$$
 (X_1^2 - aY_1^2)(X_2^2 - bY_2^2)(X_3^2 - abY_3^2) = c.
 $$
 Soit  $L=K(\sqrt{a},\sqrt{b})$. Soit $E'$ la $K$-vari\'et\'e
 d'\'equation
 $$\Norm_{L/K}(\Xi) = c^2.$$
 
Alors :

(i) Pour toute valuation 
discr\`ete 
 $\gamma$ de $K$, 
de hens\'elis\'e $K_{\gamma}$, on a $E(K_{\gamma}) \neq \emptyset$
et $E'(K_{\gamma}) \neq \emptyset$.

(ii) On a $E(K)=\emptyset$ et $E'(K) =\emptyset$.

(iii) Il existe un module galoisien fini  $\mu$ sur le corps $K$
et une classe non nulle dans $H^2(K,\mu)$ qui s'annule dans chaque
$H^2(K_{\gamma},\mu)$.

(iv) Il existe un $K$-groupe semi-simple connexe $G$ et un
\'el\'ement non trivial de $H^1(K,G)$ d'image triviale dans
chaque $H^1(K_{\gamma},G)$.
\end{cor}

\begin{proof}
On commence par observer que   tout
anneau de valuation discr\`ete  $T$  du corps des
fonctions de $\mathcal X$ contient le groupe 2-divisible $R^{\times}$
et donc l'anneau
 $R$, donc, puisque 
${\mathcal X}/R$ est projectif,
d\'efinit un morphisme dominant $\Spec T \to \mathcal{X}$.

Les \'enonc\'es pour $E$ r\'esultent imm\'ediatement
de ce qui pr\'ec\`ede.  Ceux pour $E'$ et pour $\mu$ r\'esultent des rappels du
paragraphe \ref{algebre}, selon lesquels :

(a) on a une injection 
$$H^1(K,Q) \hookrightarrow H^1(K,T),$$
envoyant  la classe de $E$ sur celle de $E'$,
et fonctorielle en le corps de base $K$.

(b)  Pour $\mu$ un module galoisien fini  convenable, 
on a une injection
$$H^1(K,Q) \hookrightarrow H^2(K,\mu)$$
 fonctorielle en le corps de base $K$.

Une fois \'etabli l'\'enonc\'e (iii), il suffit de recopier
l'argument de Serre (\cite[Chap. III, \S 4.7]{serreCG}) pour obtenir l'\'enonc\'e (iv).
L'\'enonc\'e (iii) est en effet l'analogue dans notre contexte du lemme 8
de Serre (lemme d'arithm\'etique  \og nettement moins trivial \fg). 
Le lemme 9 de \cite{serreCG} vaut aussi dans notre contexte : le groupe not\'e $S$ est une restriction
\`a la Weil d'un produit de groupes $SL_{n}$, on a donc $H^1(F,S)=0$
sur tout surcorps $F$ du corps de base. La d\'emonstration du th\'eor\`eme 8 de \cite{serreCG}
 est formelle \`a partir des  lemmes 8 et 9 de \cite{serreCG}
\end{proof}

\begin{rema}\label{25suffitpas}
 Soit $R=\C[[t]]$. Dans la situation du corollaire \ref{contrexcourbe}, montrons que 
l'on peut choisir les $\pi_{i}$ dans la proposition \ref{contreexempleglobalgeneral}
de telle fa\c con 
que l'obstruction d\'efinie au \S \ref{recipweil}, provenant de la loi
 de r\'eciprocit\'e de Weil sur la courbe $X={\mathcal X}\times_{R}K$,
 ne permette pas d'\'etablir $E(K) \neq \emptyset$.
 Dans le cas ici consid\'er\'e, le seul groupe de cohomologie 
 qui puisse, via la proposition \ref{weilrecip},  donner une information est $H^2_{nr}(K(E)/K, \Z/2)$,
 groupe qui est engendr\'e modulo $H^2(K,\Z/2)$
 par la classe de l'alg\`ebre de quaternions $A$.

  Avec les notations ci-dessus, les 
diviseurs des $\pi_{i}$ sur la fibre g\'en\'erique $X$ sont deux \`a deux disjoints.

On calcule les valeurs possibles pour $$\partial_{v}(A(P_{v})) \in H^1(\kappa(v),\Z/2)=\kappa(v)^{\times}/\kappa(v)^{\times 2} = \Z/2$$
pour $v \in X^{(1)}$ et $P_{v} \in E(K_{v})$.

En utilisant la proposition \ref{valeursdeA}, on trouve :

Si toutes les valuations $v(\pi_{i})$ sont paires, on obtient   $0$.

Si $v(\pi_{1})$ est impair, on trouve $\overline{\pi}_{2} $ et    $\overline{\pi}_{3}$.

  Si $v(\pi_{2})$ est impair, on trouve $0$ et $\overline{\pi}_{1} \overline{\pi}_{3 }$

 Si $v(\pi_{3})$ est impair, on trouve $0$ et $\overline{\pi}_{1} \overline{\pi}_{2}$.
 
Pour tout $v \in X^{(1)}$, la corestriction
 $\Z/2=\kappa(v)^{\times}/\kappa(v)^{\times 2} \to k^{\times}/k^{\times 2}=\Z/2$
est l'identit\'e.
Passant en notation additive, on voit que l'on a
 $$\sum_{v \in X^{(1)}} \Cores_{\kappa(v)/k}(\partial_{v}(A(P_{v}))) = 1  \in  \Z/2$$
pour toute ad\`ele $\{P_{v}\}$ si et seulement si  l'on a les trois propri\'et\'es 
suivantes :

Si  $v(\pi_{1})$ est impair  on a  $\overline{\pi}_{2} = \overline{\pi}_{3}$,
et il existe un nombre impair de tels $v$ avec  $\overline{\pi}_{2} = \overline{\pi}_{3} $
non carr\'e dans $\kappa(v)^{\times}$.

 Si $v(\pi_{2})$ est impair, $\overline{\pi}_{1}= \overline{\pi}_{3 }$.
 
 Si $v(\pi_{3})$ est impair, $ \overline{\pi}_{1}= \overline{\pi}_{2}$.
 
Or on peut choisir les $\pi_{i}$ dans la construction ci-dessus de fa\c con
qu'il y ait une valuation   $v \in X^{(1)}$ avec $v(\pi_{2})$  impair
et $\overline{\pi}_{1} \neq \overline{\pi}_{3 }$.  Voici comment. Une fois choisi $\pi_{1}$,
on choisit une place $v$ quelconque non dans le diviseur de $\pi_{1}$
sur $X$. On choisit $\pi_{2}$ comme ci-dessus, avec la condition suppl\'ementaire
d'avoir  $v(\pi_{2})=1$. Ceci est possible car le groupe de Picard d'un anneau semi-local est nul.
Le corps r\'esiduel $L$ en $v$ est le corps des fractions d'une extension finie de $K$,
on a $L^{\times}/L^{\times 2} = \Z/2$. On choisit alors $\pi_{3}$ comme ci-dessus,
avec la condition suppl\'ementaire que sa classe dans $L^{\times}/L^{\times 2} $
diff\`ere de celle de $\pi_{1}$. C'est possible par l'ind\'ependance des valuations
sur  le corps des fractions d'un anneau de Dedekind.

 Il  existe alors une famille $\{P_{v}\}$  telle que
 $$\sum_{v \in X^{(1)}} \Cores_{\kappa(v)/k}(\partial_{v}(A(P_{v}))) = 0 \in  \Z/2.$$
 
 La proposition  \ref{weilrecip}   ne permet donc pas de montrer $E(K)=\emptyset$,
 r\'esultat que nous obtenons en utilisant les obstructions de r\'eciprocit\'e sup\'erieure
 sur la surface $\mathcal X$.

  \end{rema}

 \begin{lem}\label{tate}
 Soit $R$ un anneau de valuation discr\`ete de corps r\'esiduel $k$
 un corps parfait
de caract\'eristique diff\'erente de $2$. Soient $K$ le corps des fractions
de $R$ et $\pi$ une uniformisante.
Le mod\`ele propre minimal r\'egulier ${\mathcal X}/R$
de la courbe elliptique  sur $K$ d'\'equation affine
$$ y^2=x^3+x^2+\pi^3$$
est de type $I_{3}$, la fibre sp\'eciale est r\'eduite, elle est
 r\'eunion de trois courbes isomorphes \`a ${\mathbf  P}^1_{k}$
 se coupant deux \`a deux transversalement en un unique point $k$-rationnel.
  \end{lem}
  \begin{proof}
On applique  l'algorithme de Tate \cite[Chap.IV, \S 9]{silverman}.
  \end{proof}

\begin{ex}\label{exconcretcourbe}

(a) Soient $R=\C[[t]]$  et $F=\C((t))$.
Le mod\`ele r\'egulier minimal  ${\mathcal X}/\C[[t]]$  de la courbe elliptique sur $\C((t))$ d'\'equation affine
$$ y^2=x^3+x^2+t^3,$$
est d'apr\`es le lemme ci-dessus de type $I_{3}$.
Cela fournit un exemple concret pour le corollaire \ref{contrexcourbe}.

On peut  aussi partir d'une courbe elliptique sur $\C((t))$ dont le mod\`ele r\'egulier minimal
dans la classification de Kodaira--N\'eron (cf. \cite[Chap. IV, Thm. 8.2]{silverman})
est de type $I_{1}$, resp. $I_{2}$,  et \'eclater deux fois, resp. une fois, 
en des points convenables pour obtenir que la fibre sp\'eciale r\'eduite est
un triangle.

(b)
Soit $K'$ le corps des fonctions d'un sch\'ema $\mathcal X$ comme en (a).
C'est une extension finie (de degr\'e 2) du corps $K=\C((t))({\mathbf P}^1)=\C((t))(x)$.
Soit $G'$ un $K'$-groupe alg\'ebrique lin\'eaire. Le descendu \`a la Weil  
$G=R_{K'/K} G'$ est un $K$-groupe lin\'eaire.
Si le $K'$-groupe $G'$ est commutatif, resp. fini, resp. un tore, resp. semi-simple,
resp. semi-simple simplement connexe, les propri\'et\'es correspondantes
sont satisfaites pour le $K$-groupe $G$. On a une bijection fonctorielle en le corps
de base 
$ H^1(K',G') \simeq H^1(K, G) $
qui \`a la classe d'un $K'$-espace principal homog\`ene $E'$ sous $G'$
associe la classe du $K$-espace principal homog\`ene $E=R_{K'/K}E'$
sous $G$. Pour $G'$ commutatif, on a  
$H^{r}(K',G') \simeq H^{r}(K, G) $ pour tout entier $r \geq 0$.
Les valuations discr\`etes  sur $K$ entretiennent  avec  celles de $K'$
les relations usuelles. 
Ce proc\'ed\'e  bien connu permet  de d\'eduire de tout contre-exemple au principe local-global
par rapport aux hens\'elis\'es en les valuations discr\`etes   de $F'$ pour un $F'$-groupe 
d'un  type  donn\'e (tore, semismple, fini, commutatif)
un contre-exemple \`a ce m\^eme principe
 sur $K=\C((t))(x)$, pour un 
 $K$-groupe convenable de m\^eme type.
  \end{ex}

\begin{cor}\label{contrexanneaulocal}
Soit $R$ un anneau local normal, hens\'elien, excellent de dimension 2,
 \`a corps r\'esiduel alg\'ebriquement clos de caract\'eristique diff\'erente de 2.
Soit $p : {\mathcal X} \to \Spec R$ une d\'esingularisation de $R$,
c'est-\`a-dire que $ {\mathcal X}$ est r\'egulier et
 le morphisme $p$ est projectif et birationnel.
Supposons que le diviseur r\'eduit associ\'e \`a la fibre  
de $p$ au-dessus de l'id\'eal maximal de $R$ contient un triangle.
Choisissons les $\pi_{i} \in  K$  et $a,b,c \in K$  comme \`a
la proposition \ref{contreexempleglobalgeneral}.
Soit $E$ la $K$-vari\'et\'e d'\'equation
$$
 (X_1^2 - aY_1^2)(X_2^2 - bY_2^2)(X_3^2 - abY_3^2) = c.
 $$

Soit  $L=K(\sqrt{a},\sqrt{b})$. Soit $E'$ la $K$-vari\'et\'e
 d'\'equation
 $$\Norm_{L/K}(\Xi) = c^2.$$
 
Alors :

(i) Pour toute valuation discr\`ete $\gamma$   sur $K$,
de hens\'elis\'e $K_{\gamma}$, on a $E(K_{\gamma}) \neq \emptyset$
et $E'(K_{\gamma}) \neq \emptyset$.

(ii) On a $E(K)=\emptyset$ et $E'(K) =\emptyset$.

(iii) Il existe un module galoisien fini  $\mu$ sur le corps $K$
et une classe non nulle dans $H^2(K,\mu)$ qui est triviale dans  $H^2(K_{\gamma},\mu)$
pour chaque valuation discr\`ete $\gamma$   de $K$.

(iv) Il existe un $K$-groupe semi-simple connexe $G$ et un
\'el\'ement non trivial de $H^1(K,G)$ d'image triviale dans
  $H^1(K_{\gamma},G)$ pour chaque valuation discr\`ete $\gamma$   de $K$.
\end{cor}

\begin{proof}
On commence par observer que   tout
anneau de valuation discr\`ete   du corps des fractions de $R$
contient le groupe infiniment 2-divisible $R^{\times}$, donc aussi l'anneau $R$.
Le reste de la d\'emonstration 
  est enti\`erement analogue \`a la d\'emonstration pr\'ec\'edente.
\end{proof}

\begin{ex}\label{exconcretlocal}

 Rappelons un exemple concret
fourni par Gabber (\cite[\S 3.4, Remark 2]{CTGiPa}) : on prend $R=\C[[x,y,z]]/(xyz+x^4+y^4+z^4)$.
On peut v\'erifier que dans la fibre d'une r\'esolution convenable $\mathcal X \to \Spec R$ au-dessus
du point singulier $(x,y,z)=(0,0,0)$ on a un triangle. Notant $K$ le corps des fractions de $R$.
On peut v\'erifier que la $K$-vari\'et\'e d'\'equation
$$ (X_1^2 - yzY_1^2)(X_2^2 - xzY_2^2)(X_3^2 - xyY_3^2) = xyz(x+y+z) $$
a des points dans tous les hens\'elis\'es de $K$ et n'a pas de point dans $K$.

 \medskip

On a dans ce cas une inclusion $\C[[x,y]]  \subset R$,
et une extension finie de corps des fractions 
$K/ \C((x,y))$.
En utilisant un argument de restriction du corps de base \`a la Weil, comme 
\`a la remarque  \ref{exconcretcourbe}, on construit
des exemples sur le corps $\C((x,y))$.
\end{ex}

\medskip

Le corollaire \ref{contrexanneaulocal} et l'exemple \ref{exconcretlocal}
r\'epondent \`a la question pos\'ee
\`a la fin du~\S  3 de \cite{CTGiPa}. On notera que pour un $K$-groupe fini  commutatif
d\'eploy\'e $\mu$ d'ordre inversible dans $R$, l'application
$$H^2(K,\mu) \to \prod_{\gamma} H^2(K_{\gamma},\mu)$$
 est injective
pour $K$ et  les valuations $\gamma$  comme au corollaire \ref{contrexcourbe} 
ou au corol\-laire~\ref{contrexanneaulocal}.  C'est une cons\'equence imm\'ediate de la propri\'et\'e
du  groupe
de Brauer de $K$ cit\'ee au d\'ebut de l'introduction.

\subsection{Exemples au-dessus d'un anneau de valuation discret complet \`a  corps r\'esiduel
non alg\'ebriquement clos}

L'hypoth\`ese que le corps r\'esiduel des points ferm\'es est s\'eparablement clos
nous a permis de donner une construction relativement simple et ind\'ependante
de la nature de la surface $\mathcal X$. 

Nous rel\^achons ici l'hypoth\`ese sur les corps r\'esiduels aux points ferm\'es,
en nous
 limitant  au cas \og semi-global \fg.

\begin{prop}\label{petitcorps}
 Soient $R$ un anneau de valuation discr\`ete 
complet
 de corps r\'esiduel $k$ parfait
  de caract\'eristique diff\'erente de 2, $F$ son corps des fractions
   et ${\mathcal X}/R$ un $R$-sch\'ema projectif r\'egulier
  int\`egre \`a fibre  g\'en\'erique une courbe
   lisse et g\'eom\'etriquement connexe.
  Soit $K$ son corps des fonctions.
   Supposons que la fibre
  sp\'eciale $X_{0}/k$ est une union transverse de trois courbes $L_{i}, i=1,2,3$,
   chacune $k$-isomorphe \`a ${\mathbf P}^1_{k}$,
  formant triangle,
  les 3 points d'intersection \'etant $k$-rationnels.
  
Quitte \`a remplacer $R$ par une extension finie \'etale, il existe 
des \'el\'ements $a,b,c \in K$ tels que
la $K$-vari\'et\'e d\'efinie par
$$
 (X_1^2 - aY_1^2)(X_2^2 - bY_2^2)(X_3^2 - abY_3^2) = c.
 $$
 n'ait pas de point dans $K$, mais ait des points dans
 tous les compl\'et\'es de $K$ pour les valuations discr\`etes 
 de ce corps.
 \end{prop}
 
 \begin{proof} 
  Le foncteur $R_{1} \to R_{1}\otimes_{R}k = k_{1}$ est une \'equivalence de cat\'egories
 entre les extensions finies \'etales d'alg\`ebres locales hens\'eliennes $R_{1}/R$ et
 les extensions finies s\'eparables de corps $k_{1}/k$ (\cite[Corollaire, p. 84]{raynaud}).

  Pour $R_{1}/R$ de ce type, 
 le sch\'ema ${\mathcal X}\times_{R}R_{1}$ est un sch\'ema r\'egulier int\`egre,
 de fibre sp\'eciale $X_{0}\times_{k}k_{1}$ admettant exactement la m\^eme configuration
 que $X_{0}/k$.
 
 Soit  $A$ le semi-localis\'e en $m_1, m_2, m_3,$ et soient, par abus de langage,
$m_1, m_2, m_3$  les id\'eaux maximaux associ\'es. Soit $\pi_1
\in A$ satisfaisant 
${\rm div}_{\mathcal X}(\pi_1) =  L_1 + D_1$,   le support de $D_1$
ne contenant  aucun de
$m_1, m_2,  m_3$. Soit  $\pi_2 \in A$ tel
${\rm div}_{\mathcal X}(\pi_2) =  L_2 +  D_2$,
le support de $D_{2}$ sans composante commune avec celui de $D_{1}$,
 ne contenant aucun des $m_{i}$
ni aucun des points d'une intersection  $L_i \cap Supp(D_1)$.

Soit $\pi_3 \in A$ tel que
${\rm div}_{\mathcal X}(\pi_3) = L_3 + D_3$,   le support de $D_3$
sans composante commune avec celui de $D_{1}$ ou de $D_{2}$,
ne contenant aucun des $m_{i}$  ni aucun des points d'une intersection 
 $L_i
\cap Supp(D_1)$ ou $L_i \cap Supp( D_2)$.

Alors   $\pi_1, \pi_2, \pi_3
\in A$ sont des \'el\'ements premiers d\'efinissant respectivement  $L_1, L_2, L_3$.
Dans $A$, on a
$m_1 = (\pi_2, \pi_3)$, $m_2 = (\pi_1, \pi_3) $ and $m_3 = (\pi_1,
\pi_2)$. 

Soit  $a = \pi_2\pi_3$, $b = \pi_3\pi_1$ et  $c =
\pi_1\pi_2\pi_3$.  Soit $E$  la $K$-vari\'et\'e  d\'efinie par
$$(X_1^2 - aY_1^2)(X_2^2 - bY_2^2)(X_3^2 - abY_3^2) = c.$$

Notons $\overline{k}$ une cl\^oture alg\'ebrique de $k$.
Pour chaque permutation $(i,j,\ell)$ 
de $(1,2,3)$, il existe une extension finie $k_{i}/k$
telle que la classe de l'alg\`ebre de quaternions d\'efinie par $(\pi_{j},\pi_{\ell})$ sur le corps des fonctions 
$k(L_{i})$ s'annule sur $k_{i}(L_{i})$. 
Le groupe de Brauer de ${\overline k}(L_{i})$
est en effet nul (th\'eor\`eme de Tsen).

Soit  $ T$ l'ensemble fini de points ferm\'es 
$$T= \cup _{i,j} (L_i \cap Supp(D_j)) \cup   \{ m_1, m_2, m_3 \}.$$
Pour chaque $P \in T$, l'un  au moins des $\pi_{i}$ est une unit\'e en $P$.
On choisit un tel $\pi_{i}$, et on pose $u_{P}=\pi_i(P) \in \kappa(P)^{\times}$.

Soit $k'$ une extension galoisienne finie de $k$  contenant tous
 les  corps $k_i$ et tous les
corps $\kappa(P)(\sqrt{u_P})$ for   $P \in T$, et aussi $\sqrt{-1}$.

Rempla\c cons  $R$ par l'extension finie \'etale locale  $R'/R$  de corps r\'esiduel $k'$,
et $F$ par le corps des fractions $F'$ de $R'$.

Quitte \`a changer les notations, sur ce nouvel  anneau, renomm\'e $R$,
de corps des fractions renomm\'e $F$,
on a $k=k_{i}$ pour tout $i$ et  $k=\kappa(P)(\sqrt{u_P})$  pour tout $P\in T$.

 Comme la preuve de la proposition \ref{semilocal} ne d\'epend pas du corps r\'esiduel $k$,
 elle donne $E(K)=\emptyset$.

Montrons que $E$ poss\`ede des points dans tout hens\'elis\'e $K_{v}$
de $K$ en une valuation $v$ discr\`ete.
Soit $x \in {\mathcal X}$ le centre de la valuation.

Supposons d'abord que  le point $x$ 
est   le point g\'en\'erique de l'une des composantes $L_{i}$ de la fibre sp\'eciale $X_{0}$.
Supposons que $x$ est le point g\'en\'erique de $L_{1}$.
L'alg\`ebre de quaternions $(\pi_{2},\pi_{3})$ est d\'eploy\'ee sur le corps $k(L_{1})$.
Il en est donc de m\^eme de l'alg\`ebre $(\pi_{2}\pi_{3},-\pi_{2})$.
L'\'equation
$$ X^2-\pi_{2}\pi_{3}Y^2=-\pi_{2}$$
a donc une solution $ (u,v)$  dans le hens\'elis\'e de l'anneau local en $x$,
et $E$ poss\`ede donc le point rationnel $(u,v,0,1,1,0)$ sur $K_{x}$.
Le calcul aux points g\'en\'eriques de $L_{2}$ et $L_{3}$ est le m\^eme.

Supposons que $x$ est un point ferm\'e $M$ de $\mathcal X$, donc un point ferm\'e
de la fibre sp\'eciale.
 
 Supposons d'abord $M \notin T$. Le point $M$ appartient \`a l'une des composantes $L_{i}$
 de la fibre sp\'eciale. Supposons que ce soit $L_{1}$. Alors $\pi_{2}$ et $\pi_{3}$ sont des unit\'es en $M$.
 L'alg\`ebre de quaternions $(\pi_{2},\pi_{3})$ est donc non ramifi\'ee au voisinage de $M$.
 Dans le corps r\'esiduel de $L_{1}$, sa classe  est nulle. Ceci implique que sa classe est nulle dans le corps
 r\'esiduel de $M$, et donc aussi dans le groupe de Brauer du hens\'elis\'e
 de l'anneau local de ${\mathcal X}$ en $M$.
L'\'equation
$$ X^2-\pi_{2}\pi_{3}Y^2=-\pi_{2}$$
a donc une solution $ (u,v) $ dans $K_{M}$,
et $E$ poss\`ede donc le point rationnel $(u,v,0,1,1,0)$ sur $K_{M}$.
Le calcul pour $M$ dans  $L_{2}$ ou $L_{3}$ est le m\^eme.

Supposons   $M \in T$, donc $\kappa(M)=k$,  et $M \in L_{1}$. Alors $\pi_{2}$ ou $\pi_{3}$
est une unit\'e en $M$, et pour l'un au moins d'entre eux, disons que ce soit $\pi_{2}$,
la valeur
$\pi_{2}(M) \in k^{\times}$ est un carr\'e non nul, donc $\pi_{2}$ est un carr\'e dans le hens\'elis\'e
en $M$ et donc dans $K_{M}$. 
La $K$-vari\'et\'e $E$ admet alors un point \'evident de la forme $(0,u,1,0,0,1)$ dans $K_{M}$.
Les autres cas sont analogues.

Si le point $x$ est de codimension 1 et situ\'e sur la fibre g\'en\'erique $X={\mathcal X}\times_{R}F$,
comme $R$ est hens\'elien, l'adh\'erence de $x$ sur ${\mathcal X}$ rencontre la fibre sp\'eciale
en un unique point ferm\'e, qu'on notera $M$. Cette adh\'erence est le spectre d'un
anneau local hens\'elien $S$, fini sur $R$, 
de corps r\'esiduel $\kappa(M)$.
On reprend alors la discussion ci-dessus. 

Si l'on a  $M \notin T$, et $M \in L_{1}$,
 l'alg\`ebre de quaternions $(\pi_{2},\pi_{3})$ est non ramifi\'ee au voisinage de $M$,
 donc en particulier sur $\Spec S$, et sa classe est nulle dans le corps r\'esiduel de $S$,
 donc aussi dans $\Br S$, donc dans $\Br \kappa(x)$.
 L'\'equation
$$ X^2-\pi_{2}\pi_{3}Y^2=-\pi_{2}$$
a donc une solution $ (u,v) $ dans $K_{x}$,
et $E$ poss\`ede donc le point rationnel $(u,v,0,1,1,0)$ sur $K_{x}$.
Le calcul pour $M$ dans  $L_{2}$ ou $L_{3}$ est le m\^eme.

Si l'on a $M \in T$, donc $\kappa(M)=k$,  et $M \in L_{1}$.
Alors par exemple $\pi_{2}$ est une unit\'e en $M$
et un carr\'e dans le corps r\'esiduel, il induit donc une unit\'e sur $S$
qui est un carr\'e dans le corps r\'esiduel de $S$, c'est donc un carr\'e dans $S$
et dans le corps des fractions de $S$.
La $K$-vari\'et\'e $E$ admet alors un point \'evident de la forme $(0,u,1,0,0,1)$ dans $K_{M}$.
Les autres cas sont analogues.
\end{proof}

\begin{rema}
On peut ainsi de multiples fa\c cons produire une courbe projective et lisse $X$ sur un corps $p$-adique,
de corps des fonctions $K$,
et une $K$-vari\'et\'e $E$, espace homog\`ene d'un $K$-tore, telle
que $E$ ait des points dans tous les compl\'et\'es de $K$
en ses valuations discr\`etes, mais n'ait pas de $K$-point.
On fabrique un exemple concret en partant de la courbe sur le corps $\Q_{p}$ ($p\neq 2$)
d'\'equation
$$y^2=x^3+x^2+p^3$$
dont le mod\`ele propre minimal r\'egulier sur $\Z_{p}$ est d'apr\`es le lemme \ref{tate} du type voulu,
et en passant \`a une extension finie non ramifi\'ee convenable de $\Q_{p}$ suivant la proc\'edure
d\'ecrite dans la proposition \ref{petitcorps}.

 En utilisant un
  argument de restriction du corps de base \`a la Weil, comme 
  dans l'exemple
 \ref{exconcretcourbe} (b), 
 sur tout corps $\Q_{p}(x)$,  avec $p\neq 2$, on peut donner des exemples 
 de tores et  de groupes semi-simples ayant un espace principal homog\`ene
 qui a des points dans tous les hens\'elis\'es de $\Q_{p}(x)$ par rapport aux valuations discr\`etes,
 mais qui n'a pas de point sur $\Q_{p}(x)$.
  
\end{rema}

\subsection{Traduction  des exemples   \`a la Harbater--Hartmann--Krashen}\label{traduction}

Soient $R$ un anneau de valuation discr\`ete complet, $k$ son corps r\'esiduel
et $F$ son corps des fractions. Soit $K$ le corps des fonctions d'une courbe
projective, lisse, g\'eom\'etriquement int\`egre sur $F$. Soit 
${\mathcal X}$  un sch\'ema r\'egulier int\`egre de dimension 2, propre et
 plat sur  Spec$(R)$, de corps des fonctions $K$.
 Soit $X_{0}$ la fibre sp\'eciale r\'eduite de ${\mathcal X}$ 
 et soit  ${\mathcal P}$ un ensemble fini de points ferm\'es de 
 $X_{0}$ qui contient tous les points singuliers de $X_{0}$.
 Soit
  ${\mathcal U}$  l'ensemble des composantes connexes de
$X_{0} \setminus {\mathcal P}$.  Pour tout point
 $P$ de ${\mathcal X}$, notons 
  $K_P$ le corps des fractions du compl\'et\'e de l'anneau local 
  de $\mathcal X$ en $P$. Soit $t \in R$ une uniformisante.
  
Pour $U \in {\mathcal U}$,  notons
$R_U$ le sous-anneau de $K$ form\'e des fonctions rationnelles sur $\mathcal X$
qui sont r\'eguli\`eres aux points de $U $. Soit 
$K_U$ le corps des fractions de la compl\'etion $t$-adique de 
$R_U$.  

Soit $G$ un groupe alg\'ebrique lin\'eaire lisse sur $K$.
On note $$\cyr{X}_{\mathcal P}({\mathcal X}, G) = {\rm ker} [H^1(K, G) \to
\prod_{\zeta \in {\mathcal P} \cup {\mathcal U}} H^1(K_\zeta, G)] $$
et
$$
\cyr{X}_0({\mathcal X}, G) = {\rm ker} [ H^1(K, G) \to
\prod_{P \in X_{0}} H^1(K_P, G)].$$
Dans cette derni\`ere d\'efinition, $P$ parcourt tous les points de la fibre sp\'eciale,
y compris les points g\'en\'eriques des composantes.

Harbater, Hartmann et Krashen \cite[Thm. 3.7]{HHK1}  ont  montr\'e que si
 $G$ est connexe et  $K$-rationnel, alors  
$\cyr{X}_{\mathcal P}({\mathcal X}, G ) = 1$. 
Dans  \cite[Cor. 5.9]{HHK2}, ils montrent que  pour
tout groupe alg\'ebrique lin\'eaire
$G$ sur $K$,
$\cyr{X}_0 ({\mathcal X}, G) =  \cup_{\mathcal P}\cyr{X}_{\mathcal P}({\mathcal X}, G)$,
o\`u 
${\mathcal P}$ parcourt tous les ensembles finis de points ferm\'es de 
${\mathcal X}$
contenant tous les points singuliers de
$X_{0}$.

Ils montrent \'egalement \cite[Prop. 8.4]{HHK2} que l'on a  une suite exacte d'ensembles point\'es
$$1 \to \cyr{X}_0 ({\mathcal X}, G)  \to \cyr{X}(K,G) \to \prod_{P \in X_{0,0}} H^1(K_{P},G).$$
Ici $X_{0,0}$ est l'ensemble des points ferm\'es de $X_{0}$, et
$$ \cyr{X}(K,G) = \Ker [H^1(K,G) \to \prod_{v\in \Omega} H^1(K_{v},G)].$$

\begin{prop}
Il existe un anneau de valuation discr\`ete complet $R$ de corps r\'esiduel
$k$ s\'eparablement clos de caract\'eristique diff\'erente de 2,
de corps des fractions $F$, un $R$-sch\'ema r\'egulier int\`egre, propre
et plat sur $\Spec R$ de fibre g\'en\'erique une $F$-courbe projective, lisse
g\'eom\'etriquement int\`egre, de corps des fonctions $K$, et un $K$-groupe
connexe lisse $G$, non $K$-rationnel, 
et  un ensemble $\mathcal P$ de points ferm\'es comme ci-dessus
 tels qu'aucun des ensembles 
$ \cyr{X}_0 ({\mathcal X}, G) $, $\cyr{X}_{\mathcal P}({\mathcal X}, G)$ et 
 $\cyr{X}(K,G)$ ne soit r\'eduit \`a un \'el\'ement.
\end{prop}
 
 \begin{proof}
En utilisant la proposition \ref{contreexempleglobalgeneral}, le corollaire \ref{contrexcourbe} et
 les exemples \ref{exconcretcourbe}, on trouve un $K$-groupe $G$ et un \'el\'ement
 non trivial dans  $\cyr{X}(K,G)$ dont l'image dans chaque $ H^1(K_{P},G)$
 pour $P$ point ferm\'e de $\mathcal X$ est triviale (voir le point  (ii) de la proposition 
  \ref{contreexempleglobalgeneral}). L'\'enonc\'e r\'esulte alors des rappels ci-dessus.
 \end{proof}

\section{Une descente}\label{paradescente}

  Soit $k$ un corps s\'eparablement clos de caract\'eristique diff\'erente de 2.
Soit $R$ une $k$-alg\`ebre locale int\`egre, excellente, hens\'elienne,
de corps r\'esiduel $k$.
Soit  $\mathcal X$ un sch\'ema r\'egulier int\`egre de dimension 2 
\'equip\'e d'un morphisme projectif
 $ p : {\mathcal X} \to \Spec R$. On suppose que l'on est dans l'une des
 situations suivantes :

(a) L'anneau $R$ est un anneau de valuation discr\`ete,
 les fibres de $p$
sont de dimension 1, la fibre g\'en\'erique est lisse et g\'eom\'etriquement int\`egre.

(b) L'anneau  $R$ est  de dimension 2, et $p$ est birationnel.

 Soit $K$ le corps des fonctions rationnelles de $\mathcal X$. 
Soient $a,b,c \in K^{\times}$. Soit $E$ la $K$-vari\'et\'e d\'efinie par l'\'equation
 $$(X_{1}^2-aY_{1}^2)(X_{2}^2-bY_{2}^2)(X_{3}^2-abY_{3}^2)=c.$$

On suppose que le support de la r\'eunion des diviseurs de $a$ et $b$ et $c$  sur $\mathcal X$
est \`a croisements normaux stricts.

\begin{theo}\label{descente}
Soit $A=(X_{1}^2-aY_{1}^2,b) \in \Br E$.
S'il existe une famille de points locaux $P_{\gamma} \in E(K_{\gamma})$ pour $\gamma
 \in {\mathcal X}^{(1)}$ et $K_{\gamma}$ le hens\'elis\'e de $K$ en $\gamma$
telle que la famille
$\partial_{\gamma}(A(P_{\gamma}))$ pour les $\gamma \in {\mathcal X}^{(1)}$
soit dans le noyau de
$$ \oplus_{\gamma \in {\mathcal X}^{(1)}} H^1(\kappa(\gamma),\Z/2) \to \oplus_{M \in {\mathcal X}^{(2)}} \Z/2,$$
alors $E(K) \neq \emptyset$.
\end{theo}

\begin{proof}

  Si  l'un de $a$, $b$, $ab$ est un carr\'e dans $K$,
alors $E(K) \neq \emptyset$. Supposons  
donc qu'aucun de  $a$, $b$, $ab$ n'est un carr\'e dans $K$.
Soit $L = K(\sqrt{b})$. 

Par un r\'esultat d'Abhyankar \cite[Theorem 8, p.~77]{abhyankar},
quitte \`a remplacer $ {\mathcal X}$ par un \'eclat\'e,
on peut supposer que la  cl\^oture int\'egrale de 
${\mathcal  X}$ dans $L$  est un sch\'ema r\'egulier.
D'apr\`es la proposition \ref{eclat}, l'hypoth\`ese d'existence d'une famille
$\{P_{\gamma}\}$ comme dans l'\'enonc\'e est pr\'eserv\'ee par tout
\'eclatement.

 Comme  $R$ est une $k$-alg\`ebre,
ce qui garantit la validit\'e de la conjecture de Gersten pour $ \mathcal X$,
et comme le  corps $k$ est s\'eparablement clos, 
  la proposition~ \ref{blochogus} \'etablit l'existence de $\alpha \in {}_{2}\Br(K)$ 
 tel que pour tout $\gamma \in {\mathcal X}^{(1)}$, 
  on ait  
 $$\partial  _{\gamma}(\alpha)=\partial_{\gamma}(A(P_{\gamma})) \in
  H^1(\kappa(\gamma), \mathbf{Z}/2) , $$ 
  et donc, puisque $\Br \kappa(\gamma)=0$, 
 $$ \alpha= A(P_{\gamma}) \in \Br K_{\gamma}.$$

Montrons que la classe $\alpha \in \Br K$ s'annule dans $\Br L$.
Soit ${\mathcal  Z}$ la cl\^oture int\'egrale de ${\mathcal  X}$ 
dans
$L$,
qui est donc un sch\'ema r\'egulier, fini et plat sur $\mathcal X$.
Soit  $\zeta \in
{\mathcal  Z}^{(1)}$. Son image    $\gamma \in {\mathcal  X}$ 
est aussi de codimension 1.
La restriction de la valuation discr\`ete  $v_{\zeta}$ de $L$ \`a $K$
est la valuation discr\`ete  $v_{\gamma}$.  Puisque 
$\alpha = A(P_\gamma) \in \Br K_{\gamma}$ et 
 $A  = (X_1^2 - aY_1^2, b)$,  on  a $\alpha \otimes L_\zeta
= 0 \in \Br L_{\zeta}$.  Ceci vaut pour tout point $\zeta$ de codimension 1
sur la surface r\'eguli\`ere $\mathcal Z$, laquelle est propre
sur un anneau local, normal,   hens\'elien,  extension finie
de $R$, de corps r\'esiduel s\'eparablement clos de caract\'eristique diff\'erente de 2.
La proposition~\ref{blochogus} appliqu\'ee \`a $\mathcal Z$
donne alors $\alpha \otimes _KL  = 0 \in \Br L$.

Ceci implique que $\alpha \in \Br K$  est la classe d'une alg\`ebre de quaternions $(\rho,b)$
 avec $\rho \in K^{\times}$. Pour tout $\gamma \in {\mathcal X}^{(1)}$, on a
 $$ A(P_{\gamma}) =(\rho,b) \in \Br K_{\gamma}.$$

Par \'eclatements successifs  
on trouve un morphisme projectif birationnel $\mathcal Y \to \mathcal X$
tel que sur $\mathcal  Y$ la r\'eunion des  supports des diviseurs de $a,b, c$ et $\rho$ 
et aussi de la fibre sp\'eciale
forment un diviseur \`a croisements normaux dont deux composantes
irr\'eductibles 
se coupent au plus en un point.

D'apr\`es la proposition \ref{eclat},
on peut trouver sur un tel $\mathcal Y$
des points $P_{\gamma}$ compl\'etant ceux qui viennent de $\mathcal X$
et tels que la nouvelle famille $\partial_{\gamma}(A(P_{\gamma}))$ soit dans
l'homologie du complexe de Bloch-Ogus de $\mathcal Y$.
La proposition \ref{descentealgebre} appliqu\'ee \`a $\mathcal Y$ assure l'existence de $\beta \in \Br K$
d'image la famille des $\partial_{\gamma}(A(P_{\gamma}))$. 

La fonctorialit\'e covariante du complexe de Bloch-Ogus (Proposition \ref{blochogus})
donne un 
diagramme commutatif de suites exactes
\[\xymatrix{
 0    \ar[r]  & {}_{2}\Br K  \ar[d]  \ar[r]  & \bigoplus_{\gamma \in {{\mathcal Y}}^{(1)}} \hskip1mmH^1(\kappa(\gamma),\Z/2)  \ar[d]   \\
  0   \ar[r]  & {}_{2}\Br K   \ar[r]  & \bigoplus_{\gamma \in {{\mathcal X}}^{(1)}} \hskip1mmH^1(\kappa(\gamma),\Z/2)
}\]
o\`u les fl\`eches verticales de droite sont des fl\`eches d'oubli et la fl\`eche verticale
de gauche est l'identit\'e, et la commutativit\'e de ce diagramme assure $\beta=\alpha \in \Br K$.

\medskip

Pour tout $\gamma$ de codimension 1 sur $\mathcal Y$,
on a  donc
$$ (A(P_{\gamma})) =  (\rho,b)   \in \Br K_{\gamma}.$$
Notons   $f =X_{1}^2-aY_{1}^2 \in K(E)^{\times}$.
On a donc
$$(f(P_{\gamma}),b)= (\rho,b) \in \Br K_{\gamma}.$$

On peut donc r\'esoudre l'\'equation
$$f(P_{\gamma})=\rho (U^2-bV^2) \neq 0$$
sur $K_{\gamma}$.

On trouve donc que la $K$-vari\'et\'e $W$ d\'efinie par
$$X_{1}^2-aY_{1}^2 = \rho  (U^2-bV^2) \neq 0$$
et
$$(X_{1}^2-aY_{1}^2)(X_{2}^2-bY_{2}^2) = c (X_{3}^2-abY_{3}^2) \neq 0,
$$
a des solutions dans tous les $K_{\gamma}$ pour  tout $\gamma$
de codimension 1 sur  $\mathcal Y$.

 La $K$-vari\'et\'e $W$ est un torseur sur la $K$-vari\'et\'e $E$
 sous le $K$-tore
$R^1_{K(\sqrt{b})/K} \G_{m}$, le morphisme structural \'etant donn\'e
par l'oubli de la premi\`ere \'equation.

Apr\`es changement de variables,
$$(U+\sqrt{b} V)(X_{2}+\sqrt{b} Y_{2})= X_{4}+\sqrt{b} Y_{4}$$
 la $K$-vari\'et\'e $W$ est d\'efinie par le syst\`eme
$$X_{1}^2-aY_{1}^2 = \rho (U^2-bV^2) \neq 0$$
$$\rho.(X_{4}^2-bY_{4}^2) =  c (X_{3}^2-abY_{3}^2) \neq 0.$$
Ceci d\'efinit le produit de deux $K$-vari\'et\'es, chacune
un c\^one \'epoint\'e sur une quadrique lisse de dimension  2,
donn\'ee par une forme quadratique diagonale dont les
coefficients poss\`edent la propri\'et\'e : la r\'eunion de leurs
supports et du support de la fibre sp\'eciale est un diviseur
\`a croisements normaux et dont deux composantes irr\'eductibles
se coupent en au plus un point.

Montrons que pour toute valuation discr\`ete $v$ sur $K$,
une telle quadrique $q=0$ admet un point dans le hens\'elis\'e $K_{v}$.
Soit $S_{v} \subset K_{v}$ l'anneau de la valuation $v$.
Comme $k$ est s\'eparablement clos, $R$ hens\'elien et $\mathcal X \to \Spec R$
projectif,
 cet anneau est centr\'e sur un point  $x \in {\mathcal X}$. 
 On a donc une
inclusion locale $O_{\mathcal{X},x} \subset S_{v}$.
Si $x$ est un point de codimension 1 de $\mathcal X$,
l'\'enonc\'e est clair. Supposons que $x$ est un point de codimension 2.
Le corps r\'esiduel en $x$ est alors s\'eparablement clos.
Si en $x$, le quotient de deux des 4 coefficients de $q$ 
est une unit\'e \`a multiplication par un carr\'e pr\`es, alors 
dans $K_{v}$, c'est un carr\'e et la
quadrique $q=0$ a un point dans $K_{v}$. 
Si aucun de ces quotients n'est une unit\'e \`a multiplication par un carr\'e pr\`es,
l'hypoth\`ese de croisements normaux fait que l'on peut supposer que
ces 4 coefficients sont $(\pi,\delta, u\pi\delta ,v)$ avec $u,v$ des unit\'es
en $x$ et $(\pi,\delta )$ des g\'en\'erateurs de l'id\'eal maximal.
Mais la forme quadratique diagonale $<\pi,\delta , u\pi\delta ,v>$
n'a pas de z\'ero non trivial dans le hens\'elis\'e $K_{\pi}$,
car la classe de $\delta $ dans le corps des fonctions de $O_{{\mathcal X},x}/\pi$
n'est pas un carr\'e. Ce cas est donc exclu. Ainsi
 $q=0$  admet des points dans tous les
hens\'elis\'es $K_{v}$.

Dans le cas (b) (cas \og local \fg),  le th\'eor\`eme \cite[Thm. 3.1]{CTOjPa} sur les formes quadratiques
de rang 3 ou 4  garantit  l'existence d'un $K$-point sur $q=0$.
Appliquant ceci \`a chacune des deux quadriques, on obtient $W(K) \neq \emptyset$ et donc $E(K)\neq \emptyset$.

Pla\c cons-nous dans le cas (a) (cas \og semi-global \fg).
Une forme quadratique  $q=<1,a,b,abc>$ de rang 4 a un z\'ero non trivial sur
un corps $K$ si et seulement si la forme quadratique ternaire $<1,a,b>$ a un z\'ero  non trivial sur
le corps $K_{1}=K(\sqrt{c}) $. Dans la situation ici \'etudi\'ee,
 il existe un anneau de valuation discr\`ete hens\'elien $R_{1}$
\`a corps r\'esiduel s\'eparablement clos et un sch\'ema r\'egulier ${\mathcal X}_{1}$
projectif et plat sur $\Spec R_{1}$, \`a fibre g\'en\'erique g\'eom\'etriquement connexe
et lisse, dont le corps des fonctions est $K_{1}$. La forme quadratique
$<1,a,b>$ a des points dans tous les hens\'elis\'es $K_{1w}$.
Ainsi l'alg\`ebre de quaternions $(a,b) \in \Br K_{1} $ est en particulier triviale dans tous les 
$\Br K_{1y}$ pour $y$ parcourant les points de codimension 1 de ${\mathcal X}_{1}$.
Sa classe est donc dans $\Br {\mathcal X}_{1}$, et ce groupe est nul (\cite[Cor. 1.10]{CTOjPa},
Prop.  \ref{blochogus} ci-dessus).
Ainsi $<1,a,b>$  est isotrope sur $K_{1}$, et donc $q$ l'est sur $K$. 
On conclut comme ci-dessus
$W(K) \neq \emptyset$ et donc
 $E(K)\neq \emptyset$.
\end{proof}

\begin{cor}\label{arbreimpliquepoint}
 Si la fibre sp\'eciale r\'eduite est un arbre, et si
 $E(K_{\gamma}) \neq \emptyset$ pour tout $\gamma \in {\mathcal X}^{(1)}$,
 alors $E(K) \neq \emptyset$.
\end{cor}

\begin{proof}
 Soit $S$ l'ensemble fini des points ferm\'es de 
  $\mathcal X$ o\`u aucun des $a$, $b$ ou $ab$
n'est   le produit d'une unit\'e par un carr\'e dans $K$.
\`A chaque tel point $x$ sont attach\'ees deux courbes lisses int\`egres,
celles  qui sont dans les supports de la r\'eunion du diviseur de  $a$ et de $b$
et passent par $x$, d\'efinies localement en $x$ par les
  $\pi,\delta $ du paragraphe  \ref{calculobsrecip}.
  Soit $T$ l'ensemble de ces courbes.
Ces courbes sont deux \`a deux transverses. 
Comme $R$ est hens\'elien, si une  courbe de $T$  n'est pas une composante
de la fibre sp\'eciale, elle rencontre la fibre sp\'eciale en exactement un point.
De l'hypoth\`ese que la fibre sp\'eciale est un arbre et des observations ci-dessus
 r\'esulte que le graphe dont les sommets sont les courbes de $T$
 et les ar\^etes les intersections de deux telles courbes lorsqu'elles sont dans $S$
 est un arbre.
 
Sur chaque composante connexe de cet arbre, partant d'un point $\gamma$
quelconque et d'un point $P_{\gamma}$ de $E(K_{\gamma})$, on peut 
d'apr\`es la proposition \ref{toutZmod2}
choisir les
points $M_{\delta }$ sur tout sommet  voisin $\delta $  de fa\c con \`a ce que
l'on ait en l'intersection $M$ de $\gamma$ et $\delta $ la formule
$$\partial_{M}(\partial_{\gamma}(P_{\gamma})) + \partial_{M}(\partial_{\delta}(M_{\delta}))=0.$$
On continue de proche en proche, ce qui est possible car le graphe d\'efini sur $T$
est un arbre. Pour les courbes $\gamma$ qui n'appartiennent pas \`a $T$, on choisit
un point $P_{\gamma} \in E(K_{\gamma})$ quelconque.

Les calculs du paragraphe \ref{calculobsrecip}
 montrent alors que
la famille $\{\partial_{\gamma}(A(P_{\gamma}))    \}$ pour $\gamma$
parcourant les points de codimension 1 de $\mathcal X$ est 
dans l'homologie du complexe de Bloch-Ogus.

On conclut alors en utilisant le th\'eor\`eme \ref{descente}. 
\end{proof}

\begin{ex}
Dans la situation \og locale \fg, l'hypoth\`ese que la fibre sp\'eciale
forme un arbre est satisfaite 
si la singularit\'e de $R$ est rationnelle. C'est le cas par exemple
si $R$ est r\'egulier, par exemple si $R=\C[[X,Y]]$.

 Dans la situation \og semi-globale  \fg,   l'hypoth\`ese que la fibre sp\'eciale
forme un arbre est satisfaite  si la fibre g\'en\'erique est la droite projective
sur le corps des fractions de $R$. Elle est aussi satisfaite si
la fibre g\'en\'erique de ${\mathcal X}/R$ est une courbe elliptique
dont le mod\`ele de Kodaira-N\'eron sur $R$ (cf. \cite{silverman}) n'a pas de lacet,
ce qui  exclut les types $I_{n}, n \geq 1$.
\end{ex}

\begin{cor}
Supposons qu'en tout point 
 $\gamma \in  {\mathcal X}^{(1)}$ l'un de $a$, $b$ ou $ab$ est un
 carr\'e dans $K_{\gamma}$. Alors
 $E(K) \neq \emptyset$.
\end{cor}

\begin{proof}
L'hypoth\`ese implique clairement que l'on a $E(K_{\gamma}) \neq \emptyset$
pour tout $\gamma \in {\mathcal X}^{(1)}$.
D'apr\`es  la proposition  \ref{partoutcarrelocal},    l'hypoth\`ese implique aussi
que pour toute famille 
$\{P_{\gamma} \in E(K_{\gamma}) \}_{\gamma \in {\mathcal X}^{(1)}}$
  la famille
$\{ \partial_{\gamma}(A(P_{\gamma})) \}$ 
 est
 dans le noyau de
$$ \oplus_{\gamma \in {\mathcal X}^{(1)}} H^1(\kappa(\gamma),\Z/2) \to \oplus_{M \in {\mathcal X}^{(2)}} \Z/2.$$
On conclut alors en utilisant le th\'eor\`eme \ref{descente}. 
\end{proof}

\begin{rema}

Soit $K$ un corps de nombres, $a,b,c \in K^{\times}$, puis 
$E$ la $K$-vari\'et\'e d\'efinie par
 $$(X_{1}^2-aY_{1}^2)(X_{2}^2-bY_{2}^2)(X_{3}^2-abY_{3}^2)=c.$$
et $Z$ une $K$-compactification lisse de $E$.
En utilisant la suite
exacte 
$$ 0 \to  \Br K \to \bigoplus_{v \in \Omega_{K}} \Br K_{v} \to \Q/\Z \to 0$$
de la th\'eorie du corps de classes, la d\'emonstration ci-dessus,
fortement simplifi\'ee car on est en dimension 1, montre, via une descente, que
l'obstruction de Brauer-Manin au principe de Hasse pour 
$Z$ est la seule obstruction \`a l'existence d'un point rationnel.
C'est un cas particulier d'un th\'eor\`eme sur   les espaces principaux
homog\`enes de $K$-tores \cite[Cor. 8.7]{sansuc}. Ceci nous am\`ene \`a poser la question :
\medskip

{\it  Le th\'eor\`eme~\ref{descente}
est-il  un cas particulier d'un th\'eor\`eme g\'en\'eral sur les espaces principaux
homog\`enes de tores ?}

\end{rema}

\bigskip

{\bf Remerciements}. Le travail sur cet article a commenc\'e en 2010. 
Pour de nombreuses discussions, nous remercions 
David Harari, David Harbater, Yong Hu et Alena Pirutka.
Les institutions suivantes nous ont permis d'avancer dans ce projet :
l'Universit\'e Emory   (Atlanta), l'Universit\'e Paris Sud (Orsay), et  
 le Centre Interfacultaire Bernoulli 
de l'\'Ecole Polytechnique F\'ed\'erale de Lausanne.
Pour ce travail, J.-L. Colliot-Th\'el\`ene a partiellement b\'en\'efici\'e d'une aide de
l'Agence Nationale de la Recherche portant la r\'ef\'erence ANR-12-BL01-0005.
R. Parimala and  V. Suresh are partially supported by the
National Science Foundation grants DMS-1001872 and
DMS-1301785 respectively.


\begin{thebibliography}{99}

\bibitem{abhyankar}  S. Abhyankar,
Simultaneous resolution for algebraic surfaces, Amer. J. Math. {\bf 78} (1956), 761--790.


 \bibitem{BO} S. Bloch et A. Ogus,  Gersten's conjecture and the homology of schemes,
 Ann. Sc. \'Ec. Norm. Sup. 4\`eme s\'erie {\bf 7} (1974) 181--202.
 
 \bibitem{borovoi}  
 M.  Borovoi, The Brauer-Manin obstructions for homogeneous spaces with connected or abelian stabilizer. J. f\"ur die reine und angew. Math.  (Crelle) {\bf 473} (1996), 181--194. 
 
 \bibitem{BoCTSk} M. Borovoi, J.-L. Colliot-Th\'el\`ene et A. N. Skorobogatov, The elementary obstruction and homogeneous spaces,
 Duke Math. J. {\bf 141} (2008) 321--364.
 
 \bibitem{BoKu} M. Borovoi et B. Kunyavski\u{\i},   Arithmetical birational invariants of linear algebraic groups
 over two-dimensional geometric fields, avec un appendice par P. Gille, J. Algebra {\bf 276} (2004) 292--339.
 
 \bibitem{CTBarbara}  J.-L. Colliot-Th\'el\`ene,  Birational invariants, purity and the Gersten conjecture,
 in {\it K-Theory and Algebraic Geometry : Connections with Quadratic Forms and Division Algebras}, AMS Summer Research Institute, Santa Barbara 1992, ed. W. Jacob and A. Rosenberg, Proceedings of Symposia in Pure Mathematics {\bf 58.1}, A. M. S. (1995) 1--64.
 
  
 
 \bibitem{CTreel}  J.-L. Colliot-Th\'el\`ene, 
 Groupes lin\'eaires sur les corps de fonctions de courbes r\'eelles. 
 J. f\"ur die reine und angew. Mathematik (Crelle) {\bf  474} (1996) 139--167.

\bibitem{CT12}  J.-L. Colliot-Th\'el\`ene, 
Groupe de Brauer non ramifi\'e d'espaces homog\`enes de tores, 
\`a para\^{\i}tre dans le Journal de Th\'eorie des Nombres de Bordeaux.
 

\bibitem{CTGiPa} J.-L. Colliot-Th\'el\`ene, P. Gille et R. Parimala, 
Arithmetic of linear algebraic groups over two-dimensional fields,
 Duke Math. J.  {\bf 121} (2004), no. 2, 285--341. 
 
\bibitem{CTHaSk} J.-L. Colliot-Th\'el\`ene, D. Harari et A. N. Skorobogatov,  Compactification \'equivariante d'un tore
(d'apr\`es Brylinski et K\"unnemann), Expositiones Mathematicae {\bf 23} (2005) 161--170.
 

\bibitem{CTOjPa} J.-L. Colliot-Th\'el\`ene, M. Ojanguren  et R. Parimala,
Quadratic forms over fraction fields of two-dimensional henselian rings and Brauer groups of related schemes. In:  {\it Algebra, Arithmetic and Geometry, Mumbai, 2000}, R. Parimala (editor), Part I, Narosa Publishing House, 2002, 185--217.

\bibitem{CTPaSu}
J.-L. Colliot-Th\'el\`ene, R. Parimala et V. Suresh,
 Patching and local-global principles for homogeneous spaces over function fields of $p$-adic curves,
 Commentarii Mathematici Helvetici 87 (2012) 1011-1033. 



\bibitem{Ducros} A. Ducros, L'obstruction de r\'eciprocit\'e \`a l'existence de points rationnels pour certaines vari\'et\'es sur le corps des fonctions d'une courbe r\'eelle.  J. f\"ur die reine und angew. Math.  (Crelle)  {\bf 504} (1998), 73--114. 
 

\bibitem{GiSz} P. Gille et T. Szamuely, {\it  Central simple algebras and Galois cohomology},
Cambridge studies in advanced mathematics {\bf 101} (2006).

\bibitem{GBBr} A. Grothendieck, Le groupe de Brauer  I, II,  III, in {\it Dix expos\'es sur la cohomologie des
sch\'emas}, Masson et North Holland, 1968, 88--188.

\bibitem{HaSz} D. Harari et  T. Szamuely,  Local-global questions for tori over $p$-adic function fields,
{\tt{http://arxiv.org/abs/1307.4782}}, \`a para\^{\i}tre dans Journal of Algebraic Geometry.


 

\bibitem{HHK1} D. Harbater, J. Hartmann, D. Krashen, 
Applications of patching to quadratic forms and central simple algebras, Invent. math. {\bf 178} (2009) 231--263.

\bibitem{HHK2} D. Harbater, J. Hartmann, D. Krashen, Local-global principles for torsors
over arithmetic curves,  
{\tt{http://http://arxiv.org/abs/1108.3323}}, to appear in American Journal of Mathematics.
 

\bibitem{HHK3} D. Harbater, J. Hartmann, D. Krashen, Local-global principles for Galois cohomology,
Commentarii Mathematici Helvetici,
   {\bf 89} no. 1 (2014), 215--253.
 


\bibitem{Hu1} Yong Hu,  Local-global principle for quadratic forms over fraction fields of two-dimensional henselian domains, 
Ann. Inst. Fourier {\bf 62}, no. 6 (2012), 2131--2143.
 
 

\bibitem{Hu2} Yong Hu,  Division algebras and quadratic forms over fraction fields of two-dimensional
henselian domains,  Algebra \& Number Theory {\bf 7} no. 8  (2013), 1919--1952.
 

\bibitem{Hu3}  Yong Hu,  Hasse principle for simply connected groups over function fields of surfaces, 
 {\tt{http://arxiv.org/abs/1203.1075}}, \`a para\^{\i}tre  dans Journal of the Ramanujan Mathematical Society.    
 

\bibitem{JaSa} U. Jannsen et S. Saito, Kato homology of arithmetic schemes and higher 
class field theory over local fields, Doc. math. Extra volume Kato (2003) 479--538.

 

\bibitem{kato} K. Kato, A Hasse principle for two-dimensional fields, J. f\"ur die reine und angew. Math.  (Crelle)  {\bf 366} (1986)
142--181.

\bibitem{Manin} Yu. I. Manin, Le groupe de Brauer--Grothendieck en g\' eom\'etrie diophantienne,  in
{\it Congr\` es intern. math. Nice}, 1970, tome {\bf  I}, Gauthier-Villars, Paris (1971) p. 401--411.

\bibitem{panin} I. A. Panin,
Le cas \'equicaract\'eristique de la conjecture de Gersten,
   Tr. Mat. Inst. Steklova 241 (2003), Teor. Chisel, Algebra i Algebr. Geom., 169--178; trad. ang. dans
Proc. Steklov Inst. Math. 2003, no. 2 (241), 154--163.

\bibitem{parimala} R. Parimala,  Arithmetic of linear algebraic groups over two-dimensional fields, in 
{\it Proceedings of the International Congress of Mathematicians 2010}. Volume I, 339--361, Hindustan Book Agency, New Delhi, 2010. 
 

\bibitem{preeti} R. Preeti, Classification theorems for hermitian forms, the Rost kernel
and Hasse principle over fields with $\cd_{2}(k) \leq 3$, J. Algebra {\bf 385} (2013) 294--313.
 

\bibitem{raynaud} M. Raynaud, {\it Anneaux locaux hens\'eliens}, LNM {\bf 169} (1970) Springer-Verlag.

\bibitem{surGab} J. Riou, Classes de Chern,  morphismes de Gysin, puret\'e absolue, Expos\'e XVI  {\it in}
{\it  Travaux de Gabber sur l'uniformisation locale et la cohomologie \'etale des sch\'emas quasi-excellents}, S\'eminaire \`a l'\'Ecole polytechnique 2006--2008 (L. Illusie, Y. Laszlo et F. Orgogozo), \`a para\^{\i}tre.

\bibitem{sasa} S. Saito et K. Sato, A finiteness theorem for zero-cycles over $p$-adic fields,
avec un appendice par U. Jannsen : Resolutions of singularities for embedded curves, Ann. of Math.
  {\bf} 172 (2010), 1593--1639. 
 
\bibitem{sansuc} J.-J. Sansuc,  Groupe de Brauer et arithm\'etique des groupes alg\'ebriques lin\'eaires,
J. f\"ur die reine und angew. Math.  (Crelle) {\bf  327} (1981) 12--80.

\bibitem{serreCG} J-P. Serre, {\it Cohomologie galoisienne}, Springer LNM 5 (1994) Cinqui\`eme \'edition.

\bibitem{silverman} J. Silverman, {\it Advanced  Topics in the Arithmetic of Elliptic Curves}, GTM {\bf 151} Springer Verlag (1994).


 
\end{thebibliography}
\end{document}